\documentclass[times]{nmeauth}
\pdfoutput=1
\usepackage{moreverb}
\usepackage{amsthm}
\usepackage[T1]{fontenc}
\usepackage{color}
\usepackage{array}
\usepackage{textcomp}
\usepackage{amsthm}
\usepackage{amsbsy}
\usepackage{amstext}
\usepackage{amssymb}
\usepackage{graphicx}
\usepackage{setspace}
\usepackage{cite}
\usepackage{threeparttable}

\usepackage[ruled]{algorithm}
\usepackage{algpseudocode}
\usepackage{helvet}
\usepackage{float}
\usepackage[unicode,bookmarksopen,bookmarksnumbered,citecolor=black,urlcolor=black]{hyperref}
\theoremstyle{plain}
\newtheorem{thm}{\protect\theoremname}
  \theoremstyle{remark}
  \newtheorem{rem}[thm]{\protect\remarkname}
  \theoremstyle{plain}
  \newtheorem{prop}[thm]{\protect\propositionname}
  \theoremstyle{plain}
  
  \theoremstyle{plain}
  \theoremstyle{plain}
  
  \theoremstyle{plain}

  \providecommand{\corollaryname}{Corollary}
  \providecommand{\lemmaname}{Lemma}
  \providecommand{\propositionname}{Proposition}
  \providecommand{\remarkname}{Remark}
\providecommand{\theoremname}{Theorem}

\algrenewcommand\alglinenumber[1]{
    {\sf\footnotesize#1}}
\algrenewcommand\algorithmicrequire{\textbf{Precondition:}}
\algrenewcommand\algorithmicensure{\textbf{Postcondition:}}

\newcommand\BibTeX{{\rmfamily B\kern-.05em \textsc{i\kern-.025em b}\kern-.08em
T\kern-.1667em\lower.7ex\hbox{E}\kern-.125emX}}

\def\volumeyear{2013}

\begin{document}
\runningheads{S. Rahman and X. Ren}{Stochastic Sensitivity Analysis}

\title{Novel Computational Methods for High-Dimensional Stochastic Sensitivity Analysis}

\author{Sharif Rahman\corrauth\footnotemark[2], Xuchun Ren\footnotemark[3]}

\address{College of Engineering, The University of Iowa, Iowa City, Iowa 52242, U.S.A.}

\corraddr{Sharif Rahman, Department of Mechanical \& Industrial Engineering, The University of Iowa, Iowa City, Iowa 52242, U.S.A. E-mail:  rahman@engineering.uiowa.edu}

\cgsn{U.S. National Science Foundation}{CMMI-0969044}

\begin{abstract}
This paper presents three new computational methods for calculating
design sensitivities of statistical moments and reliability of high-dimensional
complex systems subject to random input. The first method represents
a novel integration of polynomial dimensional decomposition (PDD)
of a multivariate stochastic response function and score functions.
Applied to the statistical moments, the method provides mean-square convergent analytical expressions of design sensitivities of the
first two moments of a stochastic response. The second and third methods,
relevant to probability distribution or reliability analysis, exploit
two distinct combinations built on PDD: the PDD-SPA method, entailing
the saddlepoint approximation (SPA) and score functions; and the PDD-MCS
method, utilizing the embedded Monte Carlo simulation (MCS) of the
PDD approximation and score functions. For all three methods developed,
the statistical moments or failure probabilities and their design
sensitivities are both determined concurrently from a single stochastic
analysis or simulation. Numerical examples, including a 100-dimensional
mathematical problem, indicate that the new methods developed provide
not only theoretically convergent or accurate design sensitivities,
but also computationally efficient solutions. A practical example
involving robust design optimization of a three-hole bracket illustrates
the usefulness of the proposed methods.
\end{abstract}


\maketitle

\footnotetext[2]{Professor.}
\footnotetext[3]{Graduate student.}
%\vspace{-6pt}

\section{Introduction}
%\vspace{-2pt}
Stochastic sensitivity analysis plays a central role in robust and
reliability-based design optimizations (RDO and RBDO) of complex systems.
For calculating design sensitivities of a stochastic response of interest,
the finite-difference method \cite{lecuyer94} constitutes the most
general and straightforward approach, but it mandates repeated stochastic
analyses for different instances of design variables. Therefore, for
practical design optimizations, the finite-difference method is very
expensive, if not prohibitive. The two other prominent methods, the
infinitesimal perturbation analysis \cite{glasserman91} and the score
function method \cite{rubinstein93}, have been mostly viewed as competing
methods, where both stochastic responses and sensitivities can be
obtained from a single stochastic simulation. However, there are additional
requirements of regularity conditions, in particular smoothness of
the performance function or the probability measure. Both methods,
when valid, are typically employed in conjunction with crude Monte
Carlo simulation (MCS). Unfortunately, for optimization of complex
mechanical systems, where stochastic response and sensitivity analyses
are required at each design iteration, even a single MCS is impractical,
as each deterministic trial of simulation often requires expensive
finite-element or other numerical calculations \cite{rahman09}.

The dimensional decomposition is a finite, hierarchical, and convergent
expansion of a multivariate output function in terms of its input
variables with increasing dimensions \cite{hoeffding48,efron81,sobol03,rahman12}.
The decomposition ameliorates the curse of dimensionality \cite{bellman57}
to some extent by developing an input-output behavior of complex systems
with low effective dimensions \cite{caflisch97}, wherein the degrees
of interactions between input variables attenuate rapidly or vanish
altogether. Based on a coupling between dimensional decomposition
and score function, Rahman \cite{rahman09} developed an efficient
method for calculating design sensitivities of stochastic systems.
The method, which sidesteps the need for crude MCS, is capable of
estimating both the probabilistic response and its sensitivity from
a single stochastic analysis without requiring performance function
gradients. Another related method, proposed by Huang and Zhang \cite{huang12},
combines Daniel's saddlepoint approximation (SPA) \cite{daniels54}
with Xu and Rahman's dimension-reduction integration technique \cite{xu04}
to perform stochastic sensitivity analysis. In their method, the sensitivity
of reliability through SPA is connected to the sensitivities of moments
of the performance function. To calculate the sensitivities of moments,
the kernel functions, similar to the score functions, are used with
dimension-reduction integration, which is the same as the dimensional
decomposition exploited by Rahman \cite{rahman09}. Nonetheless, Huang
and Zhang's method offers a few additional advantages: the tail probabilistic
characteristics of a stochastic response, if they closely follow the
exponential family of distributions, are accurately estimated by SPA;
furthermore, the embedded MCS of Rahman \cite{rahman09} for calculating
sensitivity of reliability is avoided. It is important to clarify
that the ``dimensional decomposition'' and ``dimension-reduction''
concepts invoked by these two sensitivity methods are the same as
the referential dimensional decomposition (RDD) formally presented
in latter works \cite{rahman11,rahman12}. Therefore, both methods
essentially employ RDD for multivariate function approximations, where
the mean values of random input are treated as the reference point
\cite{xu04}. The developments of these methods were motivated by
the fact that RDD requires only function evaluations, as opposed to
high-dimensional integrals required by another dimensional decomposition,
known as the ANOVA dimensional decomposition \cite{efron81} or its
polynomial version, the polynomial dimensional decomposition (PDD)
\cite{rahman08,rahman09b}. However, a recent error analysis \cite{rahman12}
reveals sub-optimality of RDD approximations, meaning that an RDD
approximation, regardless of how the reference point is chosen, cannot
be better than an ANOVA approximation for identical degrees of interaction.
The analysis also finds ANOVA approximations to be exceedingly more
precise than RDD approximations at higher-variate truncations. Therefore,
a more precise function decomposition, such as the PDD \cite{rahman08,rahman09b},
which inherits all desirable properties of the ANOVA dimensional decomposition,
should be employed for sensitivity analysis.

This paper presents three new computational methods for calculating
design sensitivities of statistical moments and reliability of high-dimensional
complex systems subject to random input. The first method represents
a novel integration of PDD of a multivariate stochastic response function
and Fourier-polynomial expansions of score functions associated with
the probability measure of the random input. Applied to the statistical
moments, the method provides analytical expressions of design sensitivities
of the first two moments of a stochastic response. The second and
third methods, relevant to probability distribution or reliability
analysis, exploit two distinct combinations grounded in PDD: the PDD-SPA
method, entailing SPA and score functions; and the PDD-MCS method,
utilizing the embedded MCS of PDD approximation and score functions.
Section 2 describes the PDD approximation of a multivariate function,
resulting in explicit formulae for the first two moments, and the
PDD-SPA and PDD-MCS methods for reliability analysis. Section 3 defines
score functions and unveils new closed-form formulae or numerical
procedures for design sensitivities of moments. The convergence of
the sensitivities of moments by the proposed method is also proved
in this section. Section 4 describes the PDD-SPA and PDD-MCS methods
for sensitivity analysis and explains how the effort required to calculate
the failure probability also delivers its design sensitivities, sustaining
no additional cost. The calculation of PDD expansion coefficients,
required in sensitivity analyses of both moments and failure probability,
is discussed in Section 5. In Section 6, six numerical examples are
presented to probe the convergence properties, accuracy, and computational
efficiency of the proposed methods, including design optimization
of a three-hole bracket. Finally, conclusions are drawn in Section
7.

\section{Polynomial Dimensional Decomposition Methods for Stochastic Analyses}

Let $\mathbb{N}$, $\mathbb{N}_{0}$, $\mathbb{R}$, and $\mathbb{R}_{0}^{+}$
represent the sets of positive integer (natural), non-negative integer,
real, and non-negative real numbers, respectively. For $k\in\mathbb{N}$,
denote by $\mathbb{R}^{k}$ the $k$-dimensional Euclidean space and
by $\mathbb{N}_{0}^{k}$ the $k$-dimensional multi-index space. These
standard notations will be used throughout the paper.

Consider a measurable space $(\Omega,\mathcal{F})$, where $\Omega$
is a sample space and $\mathcal{F}$ is a $\sigma$-field on $\Omega$.
Defined over $(\Omega,\mathcal{F})$, let $\{P_{\mathbf{d}}:\mathcal{F}\to[0,1]\}$
be a family of probability measures, where for $M\in\mathbb{N}$ and
$N\in\mathbb{N}$, $\mathbf{d}=(d_{1},\cdots,d_{M})\in\mathcal{D}$
is an $\mathbb{R}^{M}$-valued design vector with non-empty closed
set $\mathcal{D}\subseteq\mathbb{R}^{M}$, and $\mathbf{X}:=(X_{1},\cdots,X_{N}):(\Omega,\mathcal{F})\to(\mathbb{R}^{N},\mathcal{B}^{N})$
be an $\mathbb{R}^{N}$-valued input random vector
with $\mathcal{B}^{N}$ representing the Borel $\sigma$-field on
$\mathbb{R}^{N}$, describing the statistical uncertainties in loads,
material properties, and the geometry of a complex mechanical system.
The probability law of $\mathbf{X}$ is completely defined by a family
of the joint probability density functions (PDF) $\{f_{\mathbf{X}}(\mathbf{x};\mathbf{d}),\:\mathbf{x}\in\mathbb{R}^{N},\:\mathbf{d}\in\mathcal{D}\}$
that are associated with probability measures $\{P_{\mathbf{d}},\:\mathbf{d}\in\mathcal{D}\}$,
so that the probability triple $(\Omega,\mathcal{F},P_{\mathbf{d}})$
of $\mathbf{X}$ depends on $\mathbf{d}$. A design variable $d_{k}$
can be any distribution parameter or a statistic $-$ for instance,
the mean or standard deviation $-$ of one or more random variables.

\subsection{Polynomial dimensional decomposition}

Let $y(\mathbf{X})$ be a real-valued, square-integrable, measurable
transformation on $(\Omega,\mathcal{F})$, describing the relevant
performance function of a complex system. It is assumed that $y:(\mathbb{R}^{N},\mathcal{B}^{N})\to(\mathbb{R},\mathcal{B})$
is not an explicit function of $\mathbf{d}$, although $y$ implicitly
depends on $\mathbf{d}$ via the probability law of $\mathbf{X}$.
Assuming independent coordinates of $\mathbf{X}$, its joint PDF is
expressed by a product, $f_{\mathbf{\mathbf{X}}}(\mathbf{x};\mathbf{d})={\prod_{i=1}^{i=N}}f_{X_{i}}(x_{i};\mathbf{d})$,
of marginal PDF $f_{X_{i}}:\mathbb{R}\to\mathbb{R}_{0}^{+}$ of $X_{i}$,
$i=1,\cdots,N$, defined on its probability triple $(\Omega_{i},\mathcal{F}_{i},P_{i,\mathbf{d}})$
with a bounded or an unbounded support on $\mathbb{R}$. Then, for
a given subset $u\subseteq\{1,\cdots,N\}$, $f_{\mathbf{X}_{-u}}(\mathbf{x}_{-u};\mathbf{d}):=\prod_{i=1,i\notin u}^{N}f_{X_{i}}(x_{i};\mathbf{d})$
defines the marginal density function of $\mathbf{X}_{-u}:=\mathbf{X}_{\{1,\cdots,N\}\backslash u}$.

\subsubsection{ANOVA dimensional decomposition}

The analysis-of-variance (ANOVA) dimensional decomposition, expressed
by the recursive form \cite{efron81,sobol03,rahman12}
\begin{align}
y(\mathbf{X}) & ={\displaystyle \sum_{u\subseteq\{1,\cdots,N\}}y_{u}(\mathbf{X}_{u};\mathbf{d})},\label{ANOVA1}\\
y_{\emptyset}(\mathbf{d}) & =\int_{\mathbb{R}^{N}}y(\mathbf{x})f_{\mathbf{X}}(\mathbf{x};\mathbf{d})d\mathbf{x},\label{ANOVA2}\\
y_{u}(\mathbf{X}_{u};\mathbf{d}) & ={\displaystyle \int_{\mathbb{R}^{N-|u|}}y(\mathbf{X}_{u},\mathbf{x}_{-u})}f_{\mathbf{X}_{-u}}(\mathbf{x}_{-u};\mathbf{d})d\mathbf{\mathbf{x}}_{-u}-{\displaystyle \sum_{v\subset u}}y_{v}(\mathbf{X}_{v};\mathbf{d}),\label{ANOVA3}
\end{align}
is a finite, hierarchical expansion of $y$ in terms of its input
variables with increasing dimensions, where $u\subseteq\{1,\cdots,N\}$
is a subset with the complementary set $-u=\{1,\cdots,N\}\backslash u$
and cardinality $0\le|u|\le N$, and $y_{u}$ is a $|u|$-variate
component function describing the interactive effect
of $\mathbf{X}_{u}=(X_{i_{1}},\cdots,X_{i_{|u|}})$, $1\leq i_{1}<\cdots<i_{|u|}\leq N$,
a subvector of $\mathbf{X}$. The
summation in Equation \ref{ANOVA1} comprises $2^{N}$
terms, with each term depending on a group of variables indexed by
a particular subset of $\{1,\cdots,N\}$, including the empty set
$\emptyset$.

The ANOVA component functions $y_{u}$, $\emptyset\ne u\subseteq\{1,\cdots,N\}$,
have two remarkable properties: (1) the component functions,
$y_{u}$, $\emptyset\ne u\subseteq\{1,\cdots,N\}$, have zero means;
and (2) any two distinct component functions $y_{u}$ and $y_{v}$,
where $u\subseteq\{1,\cdots,N\}$, $v\subseteq\{1,\cdots,N\}$, and
$u\neq v$, are orthogonal. Further details are available elsewhere
\cite{rahman12}.
\begin{rem}
The coefficient $y_{\emptyset}=\mathbb{E}_{\mathbf{d}}[y(\mathbf{X})]$
in Equation \ref{ANOVA2} is a function of the
design vector $\mathbf{d},$ which describes the probability distribution
of the random vector $\mathbf{X}$. Therefore, the adjective ``constant''
used to describe $y_{\emptyset}$ should be interpreted with respect
to $\mathbf{X}$, not $\mathbf{d}$. A similar condition applies for
the non-constant component functions $y_{u}$, $\emptyset\ne u\subseteq\{1,\cdots,N\}$,
which also depend on $\mathbf{d}$.
\end{rem}

\subsubsection{Orthonormal Polynomials and Stochastic Expansions}

Let $\{\psi_{ij}(x_{i};\mathbf{d});\; j=0,1,\cdots\}$ be a set of
univariate, orthonormal polynomial basis functions in the Hilbert
space $\mathcal{L}_{2}(\Omega_{i},\mathcal{F}_{i},P_{i,\mathbf{d}})$
that is consistent with the probability measure $P_{i,\mathbf{d}}$
or $f_{X_{i}}(x_{i};\mathbf{d})dx_{i}$ of $X_{i}$ for a given design
$\mathbf{d}$. For $\emptyset\ne u=\{i_{1},\cdots,i_{|u|}\}\subseteq\{1,\cdots,N\}$,
where $1\le|u|\le N$, let $(\times_{p=1}^{p=|u|}\Omega_{i_{p}},\times_{p=1}^{p=|u|}\mathcal{F}_{i_{p}},\times_{p=1}^{p=|u|}P_{i_{p},\mathbf{d}})$
be the product probability triple of $\mathbf{X}_{u}=(X_{i_{1}},\cdots,X_{i_{|u|}})$.
Denote the associated space of the $|u|$-variate component functions
of $y$ by
\begin{equation}
\mathcal{L}_{2}\left(\times_{p=1}^{p=|u|}\Omega_{i_{p}},\times_{p=1}^{p=|u|}\mathcal{F}_{i_{p}},\times_{p=1}^{p=|u|}P_{i_{p},\mathbf{d}}\right):=\left\{ y_{u}:\int_{\mathbb{R}^{|u|}}y_{u}^{2}(\mathbf{x}_{u};\mathbf{d})f_{\mathbf{X}_{u}}(\mathbf{x}_{u};\mathbf{d})d\mathbf{x}_{u}<\infty\right\} ,
\end{equation}
which is a Hilbert space. Since the joint density of $\mathbf{X}_{u}$
is separable (independence of $X_{i}\textrm{, }i\in u$), that is,
$f_{\mathbf{X}_{u}}(\mathbf{x}_{u};\mathbf{d})={\textstyle \prod_{p=1}^{|u|}}f_{X_{i_{p}}}(x_{i_{p}};\mathbf{d})$, the product $\psi_{u\mathbf{j}_{|u|}}(\mathbf{X}_{u};\mathbf{d}):=\prod_{p=1}^{|u|}\psi_{i_{p}j_{p}}(X_{i_{p}};\mathbf{d})$,
where $\mathbf{j}_{|u|}=(j_{1},\cdots,j_{|u|})\in\mathbb{N}_{0}^{|u|}$,
a $|u|$-dimensional multi-index, constitutes a multivariate orthonormal polynomial basis in $\mathcal{L}_{2}(\times_{p=1}^{p=|u|}\Omega_{i_{p}},\times_{p=1}^{p=|u|}\mathcal{F}_{i_{p}},\times_{p=1}^{p=|u|}P_{i_{p},\mathbf{d}})$.
Two important properties of these product polynomials from tensor
products of Hilbert spaces are as follows.
\begin{prop}
\label{p1}The product polynomials {\small $\psi_{u\mathbf{j}_{|u|}}(\mathbf{X}_{u};\mathbf{d})$,}
$\emptyset\ne u\subseteq\{1,\cdots,N\}$, $j_{1},\cdots,j_{|u|}\ne0$, $\mathbf{d}\in\mathcal{D}$, have zero means, i.e.,
\begin{equation}
\mathbb{E}_{\mathbf{d}}\left[\psi_{u\mathbf{j}_{|u|}}(\mathbf{X}_{u};\mathbf{d})\right]=0.\label{ep1}
\end{equation}

\end{prop}

\begin{prop}
\label{p2}Any two distinct product polynomials $\psi_{u\mathbf{j}_{|u|}}(\mathbf{X}_{u};\mathbf{d})$
and $\psi_{v\mathbf{k}_{|v|}}(\mathbf{X}_{v};\mathbf{d})$ for $\mathbf{d}\in\mathcal{D}$, where $\emptyset\ne u\subseteq\{1,\cdots,N\}$,
$\emptyset\ne v\subseteq\{1,\cdots,N\}$, $j_{1},\cdots,j_{|u|}\ne0$, $k_{1},\cdots,k_{|v|}\ne0$, are uncorrelated and each has unit variance,
i.e., \emph{
\begin{equation}
\mathbb{E}_{\mathbf{d}}\left[\psi_{u\mathbf{j}_{|u|}}(\mathbf{X}_{u};\mathbf{d})\psi_{v\mathbf{k}_{|v|}}(\mathbf{X}_{v};\mathbf{d})\right]=\left\{ \begin{array}{ll}
1 & \mathrm{if}\; u=v;\:\mathbf{j}_{|u|}=\mathbf{k}_{|v|},\\
0 & \mathrm{otherwise.}
\end{array}\right.\label{ep2}
\end{equation}
}\end{prop}
\begin{rem}
Given a probability measure $P_{i,\mathbf{d}}$ of any random variable
$X_{i}$, the well-known three-term recurrence relation is commonly
used to construct the associated orthogonal polynomials \cite{rahman09b,gautschi04}.
For $m\in\mathbb{N}$, the first $m$ recursion coefficient pairs
are uniquely determined by the first $2m$ moments of $X_{i}$ that
must exist. When these moments are exactly calculated, they lead to
exact recursion coefficients, some of which belong to classical orthogonal
polynomials. For an arbitrary probability measure, approximate methods,
such as the Stieltjes procedure, can be employed to obtain the recursion
coefficients \cite{rahman09b,gautschi04}.
\end{rem}

The orthogonal polynomial expansion of a non-constant $|u|$-variate
ANOVA component function in Equation \ref{ANOVA3} becomes \cite{rahman08,rahman09b}
\begin{equation}
y_{u}(\mathbf{X}_{u};\mathbf{d})=\sum_{{\textstyle {\mathbf{j}_{|u|}\in\mathbb{N}_{0}^{|u|}\atop j_{1},\cdots,j_{|u|}\neq0}}}C_{u\mathbf{j}_{|u|}}(\mathbf{d})\psi_{u\mathbf{j}_{|u|}}(\mathbf{X}_{u};\mathbf{d})\label{fpe}
\end{equation}
for any $\emptyset\ne u\subseteq\{1,\cdots,N\}$ with
\begin{equation}
C_{u\mathbf{j}_{|u|}}(\mathbf{d}):=\int_{\mathbb{R}^{N}}y(\mathbf{x})\psi_{u\mathbf{j}_{|u|}}(\mathbf{\mathbf{x}}_{u};\mathbf{d})f_{\mathbf{X}}(\mathbf{x};\mathbf{d})d\mathbf{x}\label{coeff}
\end{equation}
representing the corresponding expansion coefficient. Similar to $y_{\emptyset}$,
the coefficient $C_{u\mathbf{j}_{|u|}}$ also depends on the design
vector $\mathbf{d}$. When $u=\{i\}$, $i=1,\cdots,N$, the univariate
component functions and expansion coefficients are
\begin{equation}
y_{\{i\}}(X_{i};\mathbf{d})=\sum_{j=1}^{\infty}C_{ij}(\mathbf{d})\psi_{ij}(X_{i};\mathbf{d})
\end{equation}
and $C_{ij}(\mathbf{d}):=C_{\{i\}(j)}(\mathbf{d})$, respectively.
When $u=\{i_{1},i_{2}\}$, $i_{1}=1,\cdots,N-1$, $i_{2}=i_{1}+1,\cdots,N$,
the bivariate component functions and expansion coefficients are
\begin{equation}
y_{\{i_{1},i_{2}\}}(X_{i_{1}},X_{i_{2}};\mathbf{d})=\sum_{j_{1}=1}^{\infty}\sum_{j_{2}=1}^{\infty}C_{i_{1}i_{2}j_{1}j_{2}}(\mathbf{d})\psi_{i_{1}j_{1}}(X_{i_{1}};\mathbf{d})\psi_{i_{2}j_{2}}(X_{i_{2}};\mathbf{d})
\end{equation}
and $C_{i_{1}i_{2}j_{1}j_{2}}(\mathbf{d}):=C_{\{i_{1},i_{2}\}(j_{1},j_{2})}(\mathbf{d})$,
respectively, and so on. Using Propositions \ref{p1} and \ref{p2},
all component functions $y_{u}$, $\emptyset\ne u\subseteq\{1,\cdots,N\}$,
are found to satisfy the annihilating conditions of the ANOVA dimensional
decomposition. The end result of combining Equations \ref{ANOVA1}-\ref{ANOVA3}
and \ref{fpe} is the PDD \cite{rahman08,rahman09b},
\begin{equation}
y(\mathbf{X})=y_{\emptyset}(\mathbf{d})+{\displaystyle \sum_{\emptyset\ne u\subseteq\{1,\cdots,N\}}}\sum_{{\textstyle {\mathbf{j}_{|u|}\in\mathbb{N}_{0}^{|u|}\atop j_{1},\cdots,j_{|u|}\neq0}}}\!\! C_{u\mathbf{j}_{|u|}}(\mathbf{d})\psi_{u\mathbf{j}_{|u|}}(\mathbf{X}_{u};\mathbf{d}),\label{pddfull}
\end{equation}
providing a hierarchical expansion of $y$ in terms of an infinite number of coefficients and orthonormal
polynomials. In practice, the number of coefficients or polynomials
must be finite, say, by retaining at most $m$th-order polynomials
in each variable. Furthermore, in many applications, the function
$y$ can be approximated by a sum of at most $S$-variate component
functions, where $S\in\mathbb{N}$; $1\le S\le N$, resulting in the
$S$-variate, $m$th-order PDD approximation
\begin{equation}
\tilde{y}_{S,m}(\mathbf{X})=y_{\emptyset}(\mathbf{d})+{\displaystyle \sum_{{\textstyle {\emptyset\ne u\subseteq\{1,\cdots,N\}\atop 1\le|u|\le S}}}}\sum_{{\textstyle {\mathbf{j}_{|u|}\in\mathbb{N}_{0}^{|u|},||\mathbf{j}_{|u|}||_{\infty}\le m\atop j_{1},\cdots,j_{|u|}\neq0}}}\!\! C_{u\mathbf{j}_{|u|}}(\mathbf{d})\psi_{u\mathbf{j}_{|u|}}(\mathbf{X}_{u};\mathbf{d}),\label{pdd}
\end{equation}
containing $\sum_{k=0}^{S}\binom{N}{k}m^{k}$ number of PDD coefficients
and corresponding orthonormal polynomials. The inner sum of Equation \ref{pdd} contains the $\infty-$norm $||\mathbf{j}_{|u|}||_{\infty} := \max \left(j_{1},\cdots, j_{|u|}\right) \in \mathbb{N}_{0}^{|u|}$ and precludes $j_{1},\cdots, j_{|u|}\ne 0$, that is, the individual degree of each variable $X_{i}$ in $\psi_{u\mathbf{j}_{|u|}},\ i\in u$, can not be \emph{zero} since $y_{u}$ is a \emph{zero}-mean strictly $|u|-$variate function. Due to its additive structure, the approximation in Equation \ref{pdd} includes
degrees of interaction among at most $S$ input variables $X_{i_{1}},\cdots,X_{i_{S}}$,
$1\le i_{1}\le\cdots\le i_{S}\le N$. For instance, by selecting $S=1$
and $2$, the functions
\begin{equation}
\tilde{y}_{1,m}(\mathbf{\mathbf{X}})=y_{\emptyset}+{\displaystyle \sum_{i=1}^{N}}{\displaystyle \sum_{j=1}^{m}}C_{ij}(\mathbf{d})\psi_{ij}(X_{i};\mathbf{d})\label{updd}
\end{equation}
 and
\begin{equation}
\begin{array}{rcl}
\tilde{y}_{2,m}(\mathbf{X}) & = & y_{\emptyset}(\mathbf{d})+{\displaystyle \sum_{i=1}^{N}}{\displaystyle \sum_{j=1}^{m}}C_{ij}(\mathbf{d})\psi_{ij}(X_{i};\mathbf{d})+\\
 &  & {\displaystyle \sum_{i_{1}=1}^{N-1}}\;{\displaystyle \sum_{i_{2}=i_{1}+1}^{N}}{\displaystyle \sum_{j_{1}=1}^{m}}{\displaystyle \sum_{j_{2}=1}^{m}}C_{i_{1}i_{2}j_{1}j_{2}}(\mathbf{d})\psi_{i_{1}j_{1}}(X_{i_{1}};\mathbf{d})\psi_{i_{2}j_{2}}(X_{i_{2}};\mathbf{d}),
\end{array}\label{bpdd}
\end{equation}
respectively, provide univariate and bivariate $m$th-order PDD approximations,
contain contributions from all input variables, and should not be
viewed as first- and second-order approximations, nor as limiting the nonlinearity of $y$. Depending on how the component
functions are constructed, arbitrarily high-order univariate and bivariate
terms of $y$ could be lurking inside $\tilde{y}_{1,m}$ and $\tilde{y}_{2,m}$.
When $S\to N$ and $m\to\infty$, $\tilde{y}_{S,m}$ converges to
$y$ in the mean-square sense, permitting Equation \ref{pdd} to generate a hierarchical and convergent sequence of approximations
of $y$. Readers interested in further details of PDD are referred
to the authors' past works \cite{rahman08,rahman09b}.

\subsection{Statistical Moment Analysis}

Let $m^{(r)}(\mathbf{d}):=\mathbb{E}_{\mathbf{d}}[y^{r}(\mathbf{X})]$,
if it exists, define the raw moment of $y$ of order $r$, where $r\in\mathbb{N}$.
Given an $S$-variate, $m$th-order PDD approximation $\tilde{y}_{S,m}(\mathbf{X})$
of $y(\mathbf{X})$, let $\tilde{m}_{S,m}^{(r)}(\mathbf{d}):=\mathbb{E}_{\mathbf{d}}[\tilde{y}_{S,m}^{r}(\mathbf{X})]$
define the raw moment of $\tilde{y}_{S,m}$ of order $r$. The following
subsections describe the explicit formulae or analytical expressions
for calculating the moments by PDD approximations.

\subsubsection{First- and Second-Order Moments}

Applying the expectation operator on $\tilde{y}_{S,m}(\mathbf{X})$
and $\tilde{y}_{S,m}^{2}(\mathbf{X})$, and recognizing Propositions
\ref{p1} and \ref{p2}, the first moment or mean \cite{rahman10}
\begin{equation}
\tilde{m}_{S,m}^{(1)}(\mathbf{d}):=\mathbb{E}_{\mathbf{d}}\left[\tilde{y}_{S,m}(\mathbf{X})\right]=y_{\emptyset}(\mathbf{d})=\mathbb{E}_{\mathbf{d}}\left[y(\mathbf{X})\right]=:m^{(1)}(\mathbf{d})\label{mom1}
\end{equation}
of the $S$-variate, $m$th-order PDD approximation matches the exact
mean of $y$, regardless of $S$ or $m$, whereas the second moment
\cite{rahman10}
\begin{equation}
\tilde{m}_{S,m}^{(2)}(\mathbf{d}):=\mathbb{E}_{\mathbf{d}}\left[\tilde{y}_{S,m}^{2}(\mathbf{X})\right]=y_{\emptyset}^{2}(\mathbf{d})+{\displaystyle \sum_{{\textstyle {\emptyset\ne u\subseteq\{1,\cdots,N\}\atop 1\le|u|\le S}}}}\sum_{{\textstyle {\mathbf{j}_{|u|}\in\mathbb{N}_{0}^{|u|},||\mathbf{j}_{|u|}||_{\infty}\le m\atop j_{1},\cdots,j_{|u|}\neq0}}}C_{u\mathbf{j}_{|u|}}^{2}(\mathbf{d})\label{mom2}
\end{equation}
is calculated as the sum of squares of all expansion coefficients
of $\tilde{y}_{S,m}(\mathbf{X})$. Clearly, the approximate second
moment in Equation \ref{mom2} approaches the exact second moment
\begin{equation}
m^{(2)}(\mathbf{d}):=\mathbb{E}_{\mathbf{d}}\left[y^{2}(\mathbf{X})\right]=y_{\emptyset}^{2}(\mathbf{d})+{\displaystyle \sum_{\emptyset\ne u\subseteq\{1,\cdots,N\}}}\sum_{{\textstyle {\mathbf{j}_{|u|}\in\mathbb{N}_{0}^{|u|}\atop j_{1},\cdots,j_{|u|}\neq0}}}C_{u\mathbf{j}_{|u|}}^{2}(\mathbf{d})
\end{equation}
of $y$ when $S\to N$ and $m\to\infty$. The mean-square convergence
of $\tilde{y}_{S,m}$ is guaranteed as $y$, and its component functions
are all members of the associated Hilbert spaces. In addition, the
variance of $\tilde{y}_{S,m}(\mathbf{X})$ is also mean-square convergent.

For the two special cases, $S=1$ and $S=2$, the univariate and bivariate
PDD approximations yield the same exact mean value $y_{\emptyset}(\mathbf{d})$,
as noted in Equation \ref{mom1}. However, the
respective second moment approximations,
\begin{equation}
\tilde{m}_{1,m}^{(2)}(\mathbf{d})=y_{\emptyset}^{2}(\mathbf{d})+{\displaystyle \sum_{i=1}^{N}}{\displaystyle \sum_{j=1}^{m}}C_{ij}^{2}(\mathbf{d})
\end{equation}
 and
\begin{equation}
\tilde{m}_{2,m}^{(2)}(\mathbf{d})=y_{\emptyset}^{2}(\mathbf{d})+{\displaystyle \sum_{i=1}^{N}}{\displaystyle \sum_{j=1}^{m}}C_{ij}^{2}(\mathbf{d})+{\displaystyle \sum_{i_{1}=1}^{N-1}}{\displaystyle \sum_{i_{2}=i_{1}+1}^{N}}{\displaystyle \sum_{j_{2}=1}^{m}}{\displaystyle \sum_{j_{1}=1}^{m}}C_{i_{1}i_{2}j_{1}j_{2}}^{2}(\mathbf{d}),
\end{equation}
differ, depend on $m$, and progressively improve as $S$ becomes
larger. Recent works on error analysis indicate that the second-moment
properties obtained from the ANOVA dimensional decomposition, which
leads to PDD approximations, are superior to those derived from dimension-reduction
methods that are grounded in RDD \cite{rahman11,rahman12}.

\subsubsection{Higher-Order Moments}

When calculating higher-order ($2<r<\infty)$ moments by the PDD approximation,
no explicit formulae exist for a general function $y$ or the probability
distribution of $\mathbf{X}$. In which instance, two options are
proposed to estimate the higher-order moments.

Option I entails expanding the $r$th power of the PDD approximation
of $y$ by
\begin{equation}
\tilde{y}_{S,m}^{r}(\mathbf{X})=g_{\emptyset}(\mathbf{d})+{\displaystyle \sum_{{\textstyle {\emptyset\ne u\subseteq\{1,\cdots,N\}\atop 1\le|u|\le\min(rS,N)}}}}g_{u}(\mathbf{X}_{u};\mathbf{d})\label{yrPDD}
\end{equation}
in terms of a constant $g_{\emptyset}(\mathbf{d})$ and at most $\min(rS,N)$-variate
polynomial functions $g_{u}(\mathbf{X}_{u};\mathbf{d})$ and then
calculating the moment
\begin{equation}
\begin{array}{rcl}
\tilde{m}_{S,m}^{(r)}(\mathbf{d}) & = & \int_{\mathbb{R}^{N}}\tilde{y}_{S,m}^{r}(\mathbf{x})f_{\mathbf{\mathbf{X}}}(\mathbf{x};\mathbf{d})d\mathbf{x}\\
 & = & {\displaystyle g_{\emptyset}(\mathbf{d})+\sum_{{\textstyle {\emptyset\ne u\subseteq\{1,\cdots,N\}\atop 1\le|u|\le\min(rS,N)}}}\int_{\mathbb{R}^{|u|}}g_{u}(\mathbf{x}_{u};\mathbf{d})f_{\mathbf{\mathbf{X}}_{u}}(\mathbf{x}_{u};\mathbf{d})d\mathbf{x}_{u}}
\end{array}\label{momr}
\end{equation}
by integration, if it exists. For well-behaved functions, including
many encountered in practical applications, $\tilde{m}_{S,m}^{(r)}(\mathbf{d})$
should render an accurate approximation of $m^{(r)}(\mathbf{d})$,
the $r$th-order moment of $y(\mathbf{X})$, although there is no
rigorous mathematical proof of convergence when $r>2$. Note that
Equation \ref{momr} involves integrations of elementary polynomial
functions and does not require any expensive evaluation of the original
function $y$. Nonetheless, since $\tilde{y}_{S,m}(\mathbf{X})$ is
a superposition of at most $S$-variate component functions of independent
variables, the largest dimension of the integrals in Equation \ref{momr}
is $\min(rS,N)$. Therefore, Option I mandates high-dimensional integrations
if $\min(rS,N)$ is large. In addition, if $rS\ge N$ and $N$ is
large, then the resulting $N$-dimensional integration is infeasible.

As an alternative, Option II, relevant to large $N$, creates an additional
$\bar{S}$-variate, $\bar{m}$th-order PDD approximation
\begin{equation}
\tilde{z}_{\bar{S},\bar{m}}(\mathbf{X})=z_{\emptyset}(\mathbf{d})+{\displaystyle \sum_{{\textstyle {\emptyset\ne u\subseteq\{1,\cdots,N\}\atop 1\le|u|\le\bar{S}}}}}{\displaystyle \sum_{{\textstyle {\mathbf{j}_{|u|}\in\mathbb{N}_{0}^{|u|},||\mathbf{j}_{|u|}||_{\infty}\le\bar{m}\atop j_{1},\cdots,j_{|u|}\neq0}}}}\bar{C}_{u\mathbf{j}_{|u|}}(\mathbf{d})\psi_{u\mathbf{j}_{|u|}}(\mathbf{X}_{u};\mathbf{d})
\end{equation}
of $\tilde{y}_{S,m}^{r}(\mathbf{X})$, where $\bar{S}$
and $\bar{m}$, potentially distinct from $S$ and $m$, are accompanying
truncation parameters, $z_{\emptyset}(\mathbf{d}):=\int_{\mathbb{R}^{N}}\tilde{y}_{S,m}^{r}(\mathbf{x})f_{\mathbf{\mathbf{X}}}(\mathbf{x};\mathbf{d})d\mathbf{x}$,
and $\bar{C}_{u\mathbf{j}_{|u|}}(\mathbf{d}):=\int_{\mathbb{R}^{N}}\tilde{y}_{S,m}^{r}(\mathbf{x})\psi_{u\mathbf{j}_{|u|}}(\mathbf{\mathbf{x}}_{u};\mathbf{d})f_{\mathbf{X}}(\mathbf{x};\mathbf{d})d\mathbf{x}$
are the associated PDD expansion coefficients of $\tilde{z}_{\bar{S},\bar{m}}(\mathbf{X})$.
Replacing $\tilde{y}_{S,m}^{r}(\mathbf{x})$ with $\tilde{z}_{\bar{S},\bar{m}}(\mathbf{x})$,
the first line of Equation \ref{momr} produces
\begin{equation}
\tilde{m}_{S,m}^{(r)}(\mathbf{d})=\int_{\mathbb{R}^{N}}\tilde{z}_{\bar{S},\bar{m}}(\mathbf{x})f_{\mathbf{\mathbf{X}}}(\mathbf{x};\mathbf{d})d\mathbf{x}=:z_{\emptyset}(\mathbf{d}).\label{momr2}
\end{equation}
Then the evaluation of $z_{\emptyset}(\mathbf{d})$ from the definition,
which also requires $N$-dimensional integration, leads Equation \ref{momr2}
back to Equation \ref{momr}, raising the question of why Option II
is introduced. Indeed, the distinction between the two options forms
when the constant $z_{\emptyset}(\mathbf{d})$ is approximately calculated
by dimension-reduction integration, to be explained in Section 5,
entailing at most $\bar{S}$-dimensional integrations. Nonetheless,
if $\bar{S}\ll rS<N$, then a significant dimension reduction is possible
in Option II for estimating higher-order moments. In other words,
Option II, which is an approximate version of Option I, may provide
efficient solutions to high-dimensional problems, provided that a
loss of accuracy in Option II, if any, is insignificant. The higher-order
moments are useful for approximating the probability distribution
of a stochastic response or reliability analysis, including their
sensitivity analyses, and will  be revisited in the next subsection.

\subsection{Reliability Analysis}

A fundamental problem in reliability analysis entails calculation
of the failure probability
\begin{equation}
P_{F}(\mathbf{d}):=P_{\mathbf{d}}\left[\mathbf{X}\in\Omega_{F}\right]=\int_{\mathbb{R}^{N}}{\displaystyle I_{\Omega_{F}}(\mathbf{x})f_{\mathbf{X}}(\mathbf{x};\mathbf{d})d\mathbf{x}}=:\mathbb{E}_{\mathbf{d}}\left[I_{\Omega_{F}}(\mathbf{X})\right],\label{pf}
\end{equation}
where $\Omega_{F}$ is the failure set and $I_{\Omega_{F}}(\mathbf{x})$
is the associated indicator function, which is equal to \emph{one}
when $\mathbf{x}\in\Omega_{F}$ and \emph{zero} otherwise. Depending
on the nature of the failure domain $\Omega_{F}$, a component or
a system reliability analysis can be envisioned. For component reliability
analysis, the failure domain is often adequately described by a single
performance function $y(\mathbf{x})$, for instance, $\Omega_{F}:=\{\mathbf{x}:y(\mathbf{x})<0\}$.
In contrast, multiple, interdependent performance functions $y_{i}(\mathbf{x}),\; i=1,2,\cdots,$
are required for system reliability analysis, leading, for example,
to $\Omega_{F}:=\{\mathbf{x}:\cup_{i}y_{i}(\mathbf{x})<0\}$ and $\Omega_{F}:=\{\mathbf{x}:\cap_{i}y_{i}(\mathbf{x})<0\}$
for series and parallel systems, respectively. In this subsection,
two methods are presented for estimating the failure probability.
The PDD-SPA method, which blends the PDD approximation with SPA, is
described first. Then the PDD-MCS method, which exploits the PDD approximation
for MCS, is elucidated.

\subsubsection{The PDD-SPA Method}

Let $F_{y}(\xi):=P_{\mathbf{d}}[y\le\xi]$ be the cumulative distribution
function (CDF) of $y(\mathbf{X})$. Assume that the PDF $f_{y}(\xi):=dF_{y}(\xi)/d\xi$
exists and suppose that the cumulant generating function (CGF)
\begin{equation}
K_{y}(t):=\ln\left\{ \int_{-\infty}^{+\infty}\exp(t\xi)f_{y}(\xi)d\xi\right\} \label{cgf}
\end{equation}
of $y$ converges for $t\in\mathbb{R}$ in some non-vanishing interval
containing the origin. Using inverse Fourier transformation, exponential
power series expansion, and Hermite polynomial approximation, Daniels
\cite{daniels54} developed an SPA formula to approximately evaluate
$f_{y}(\xi)$. However, the success of such formula is predicated
on how accurately the CGF and its derivatives, if they exist, are
calculated. In fact, determining $K_{y}(t)$ is immensely difficult because it is equivalent to knowing all higher-order moments of $y$.
To mitigate this problem, consider the Taylor series expansion of
\begin{equation}
K_{y}(t)=\sum_{r\in\mathbb{N}}\frac{\kappa^{(r)}t^{r}}{r!}\label{cgf2}
\end{equation}
at $t=0$, where $\kappa^{(r)}:=d^{r}K_{y}(0)/dt^{r},$ $r\in\mathbb{N}$,
is known as the $r$th-order cumulant of $y(\mathbf{X})$. If some
of these cumulants are effectively estimated, then a truncated Taylor
series provides a useful means to approximate $K_{y}(t)$. For instance,
assume that, given a positive integer $Q<\infty$, the raw moments
$\tilde{m}_{S,m}^{(r)}(\mathbf{d})$ of order at most $Q$ have been
calculated with sufficient accuracy using an $S$-variate, $m$th-order
PDD approximation $\tilde{y}_{S,m}(\mathbf{X})$ of $y(\mathbf{X})$,
as described in the preceding subsection. Then the corresponding approximate
cumulants are easily obtained from the well-known cumulant-moment
relationship,
\begin{equation}
\tilde{\kappa}_{S,m}^{(r)} (\mathbf{d})= \left\{ \begin{array}{l@{\quad:\quad}l}
\tilde{m}_{S,m}^{(1)}(\mathbf{d}) & r=1,\\
\tilde{m}_{S,m}^{(r)}(\mathbf{d})-\sum\limits _{p=1}^{r-1}\binom{r-1}{p-1}\tilde{\kappa}_{S,m}^{(p)}(\mathbf{d})\tilde{m}_{S,m}^{(r-p)}(\mathbf{d})\; & 2\le r\le Q,
\end{array}\right.\label{mom2cum}
\end{equation}
where the functional argument $\mathbf{d}$ serves as a reminder that
the moments and cumulants all depend on the design vector $\mathbf{d}$.
Setting $\kappa^{(r)}=\tilde{\kappa}_{S,m}^{(r)}$ for $r=1,\cdots,Q$,
and \emph{zero} otherwise in Equation \ref{cgf2}, the result is an
$S$-variate, $m$th-order PDD approximation
\begin{equation}
\tilde{K}_{y,Q,S,m}(t;\mathbf{d})=\sum_{r=1}^{Q}\frac{\tilde{\kappa}_{S,m}^{(r)}(\mathbf{d})t^{r}}{r!}\label{cgf3}
\end{equation}
of the $Q$th-order Taylor series expansion of $K_{y}(t)$. It is
elementary to show that $\tilde{K}_{y,Q,S,m}(t;\mathbf{d})\to K_{y}(t)$
when $S\to N$, $m\to\infty$, and $Q\to\infty$.

Using the CGF approximation in Equation \ref{cgf3}, Daniel's SPA
leads to the explicit formula \cite{daniels54},
\begin{equation}
\tilde{f}_{y,PS}(\xi;\mathbf{d})=\left[2\pi\tilde{K''}_{y,Q,S,m}(t_{s};\mathbf{d})\right]^{-\frac{1}{2}}\exp\left[\tilde{K}_{y,Q,S,m}(t_{s};\mathbf{d})-t_{s}\xi\right],\label{daniels}
\end{equation}
for the approximate PDF of $y$, where the subscript "PS" stands for PDD-SPA and $t_{s}$ is the saddlepoint
that is obtained from solving
\begin{equation}
\tilde{K}'_{y,Q,S,m}(t_{s};\mathbf{d})=\xi\label{sp}
\end{equation}
with $\tilde{K}'_{y,Q,S,m}(t;\mathbf{d}):=d\tilde{K}{}_{y,Q,S,m}(t;\mathbf{d})/dt$
and $\tilde{K}''_{y,Q,S,m}(t;\mathbf{d}):=d^{2}\tilde{K}{}_{y,Q,S,m}(t;\mathbf{d})/dt^{2}$
defining the first- and second-order derivatives, respectively, of
the approximate CGF of $y$ with respect to $t$. Furthermore, based
on a related work of Lugannani and Rice \cite{lugannani80}, the approximate
CDF of $y$ becomes
\begin{equation}
\begin{array}{c}
\tilde{F}_{y,PS}(\xi;\mathbf{d})=\Phi(w)+\phi(w){\displaystyle \left(\frac{1}{w}-\frac{1}{v}\right),}\\
w=\mbox{sgn}(t_{s})\left\{ 2\left[t_{s}\xi-\tilde{K}{}_{y,Q,S,m}(t_{s};\mathbf{d})\right]\right\} ^{\frac{1}{2}},\; v=t_{s}\left[\tilde{K}''_{y,Q,S,m}(t_{s};\mathbf{d})\right]^{\frac{1}{2}},
\end{array}\label{lr}
\end{equation}
where $\Phi(\cdot)$ and $\phi(\cdot)$ are the CDF and PDF, respectively,
of the standard Gaussian variable and $\mbox{sgn}(t_{s})=+1,-1,\mbox{or}\ 0$,
depending on whether $t_{s}$ is positive, negative, or zero. According
to Equation \ref{lr}, the CDF of $y$ at a point $\xi$ is obtained
using solely the corresponding saddlepoint $t_{s}$, that is, without
the need to integrate Equation \ref{daniels} from $-\infty$ to $\xi$.

Finally, using Lugannani and Rice's formula, the PDD-SPA estimate
$\tilde{P}_{F,PS}(\mathbf{d})$ of the component failure probability
$P_{F}(\mathbf{d}):=P[y(\mathbf{X})<0]$ is obtained as
\begin{equation}
\tilde{P}_{F,PS}(\mathbf{d})=\tilde{F}_{y,PS}(0;\mathbf{d}),\label{pflr}
\end{equation}
the PDD-SPA generated CDF of $y$ at $\xi=0$. It is important to
recognize that no similar SPA-based formulae are available for the
joint PDF or joint CDF of dependent stochastic responses. Therefore,
the PDD-SPA method in the current form cannot be applied to general
system reliability analysis.

The PDD-SPA method contains several truncation parameters that should
be carefully selected. For instance, if $Q$ is too small, then the
truncated CGF from Equation \ref{cgf3} may spoil the method, regardless
of how large are $S$ and $m$ chosen in the PDD approximation.
On the other hand, if $Q$ is overly large, then many higher-order
moments involved may not be accurately calculated by the PDD approximation.
More significantly, a finite-order truncation of CGF may cause loss
of convexity of the actual CGF, meaning that the one-to-one relationship
between $\xi$ and $t_{s}$ in Equation \ref{sp} is not ensured for
every threshold $\xi$. Furthermore, the important property $\tilde{K}''_{y,Q,S,m}(t_{s};\mathbf{d})>0$
may not be maintained. To resolve this quandary, Yuen et al. \cite{yuen07}
presented for $Q=4$ several distinct cases of the cumulants, describing
the interval $(t_{l},t_{u})$, where $-\infty\le t_{l}\le0$ and $0\le t_{u}\le\infty$,
such that $t_{l}\le t_{s}\le t_{u}$ and $\tilde{K}''_{y,Q,S,m}(t_{s};\mathbf{d})>0$,
ruling out any complex values of the square root in Equation \ref{daniels}
or \ref{lr}. Table \ref{tab:1-8case} summarizes these cases, which
were employed in the PDD-SPA method described in this paper. If the
specified threshold $\xi\in(\tilde{K'}{}_{y,Q,S,m}(t_{l};\mathbf{d}),\tilde{K'}{}_{y,Q,S,m}(t_{u};\mathbf{d}))$,
then the saddlepoint $t_{s}$ is uniquely determined from Equation
\ref{sp}, leading to the CDF or reliability in Equation \ref{lr}
or \ref{pflr}. Otherwise, the PDD-SPA method will fail to provide
a solution. It is important to note that developing similar cases
for $Q>4$, assuring a unique solution of the saddlepoint, is not
trivial, and was not considered in this work.

{\small }
\begin{table}[tbph]
\centering
\begin{threeparttable}
\caption{Intervals of the saddlepoint for $Q=4$$^{\tnote{(a)}}$ }
\label{tab:1-8case}
\begin{centering}
\begin{tabular}{cccc}
\hline
Case & Condition & $t_{l}$ & $t_{u}$\tabularnewline
\hline
1 & $\tilde{\kappa}_{S,m}^{(4)}>0$, $\Delta>0$, $\tilde{\kappa}_{S,m}^{(3)}>0$  & ${\displaystyle \frac{-\tilde{\kappa}_{S,m}^{(3)}+\sqrt{\Delta}}{\tilde{\kappa}_{S,m}^{(4)}}}$ & $+\infty$ \tabularnewline
2 & $\tilde{\kappa}_{S,m}^{(4)}>0$, $\Delta>0$, $\tilde{\kappa}_{S,m}^{(3)}<0$  & $-\infty$ & ${\displaystyle \frac{-\tilde{\kappa}_{S,m}^{(3)}-\sqrt{\Delta}}{\tilde{\kappa}_{S,m}^{(4)}}}$\tabularnewline
3 & $\tilde{\kappa}_{S,m}^{(4)}>0$, $\Delta=0$  & $-\infty$$^{\tnote{(b)}}$ & $+\infty$$^{\tnote{(b)}}$\tabularnewline
4 & $\tilde{\kappa}_{S,m}^{(4)}>0$, $\Delta<0$ & $-\infty$ & $+\infty$\tabularnewline
5 & $\tilde{\kappa}_{S,m}^{(4)}=0$, $\tilde{\kappa}_{S,m}^{(3)}>0$  & ${\displaystyle -\frac{\tilde{\kappa}_{S,m}^{(2)}}{\tilde{\kappa}_{S,m}^{(3)}}}$ & $+\infty$\tabularnewline
6 & $\tilde{\kappa}_{S,m}^{(4)}=0$, $\tilde{\kappa}_{S,m}^{(3)}=0$  & $-\infty$ & $+\infty$\tabularnewline
7 & $\tilde{\kappa}_{S,m}^{(4)}=0$, $\tilde{\kappa}_{S,m}^{(3)}<0$  & $-\infty$ & ${\displaystyle -\frac{\tilde{\kappa}_{S,m}^{(2)}}{\tilde{\kappa}_{S,m}^{(3)}}}$\tabularnewline
8 & $\tilde{\kappa}_{S,m}^{(4)}<0$  & ${\displaystyle \frac{-\tilde{\kappa}_{S,m}^{(3)}+\sqrt{\Delta}}{\tilde{\kappa}_{S,m}^{(4)}}}$ & ${\displaystyle \frac{-\tilde{\kappa}_{S,m}^{(3)}-\sqrt{\Delta}}{\tilde{\kappa}_{S,m}^{(4)}}}$\tabularnewline
\hline
\end{tabular}
\begin{tablenotes}
\footnotesize
\item [(a)] For $\tilde{K}_{y,4,S,m}(t;\mathbf{d})={\displaystyle \tilde{\kappa}_{S,m}^{(1)}(\mathbf{d})t+\frac{1}{2!}\tilde{\kappa}_{S,m}^{(2)}(\mathbf{d})t^{2}+\frac{1}{3!}\tilde{\kappa}_{S,m}^{(3)}(\mathbf{d})t^{3}+\frac{1}{4!}\tilde{\kappa}_{S,m}^{(4)}(\mathbf{d})t^{4}}$,
the discriminant of $\tilde{K}'_{y,4,S,m}(t;\mathbf{d})$ is $\Delta:=\tilde{\kappa}_{S,m}^{(3)^{2}}-2\tilde{\kappa}_{S,m}^{(2)}\tilde{\kappa}_{S,m}^{(4)}$.
\item [(b)] The point $-\tilde{\kappa}_{S,m}^{(3)}/(2\tilde{\kappa}_{S,m}^{(2)})$
should not be an element of $(t_{l}, t_{u})$, i.e., $(t_{l}, t_{u})=(-\infty, \infty)\setminus \{-\tilde{\kappa}_{S,m}^{(3)}/(2\tilde{\kappa}_{S,m}^{(2)})\}$.
\end{tablenotes}
\par\end{centering}
\end{threeparttable}
\end{table}

\subsubsection{The PDD-MCS Method}

Depending on component or system reliability analysis, let $\tilde{\Omega}_{F,S,m}:=\{\mathbf{x}:\tilde{y}_{S,m}(\mathbf{x})<0\}$
or $\tilde{\Omega}_{F,S,m}:=\{\mathbf{x}:\cup_{i}\tilde{y}_{i,S,m}(\mathbf{x})<0\}$
or $\tilde{\Omega}_{F,S,m}:=\{\mathbf{x}:\cap_{i}\tilde{y}_{i,S,m}(\mathbf{x})<0\}$
be an approximate failure set as a result of $S$-variate, $m$th-order
PDD approximations $\tilde{y}_{S,m}(\mathbf{X})$ of $y(\mathbf{X})$
or $\tilde{y}_{i,S,m}(\mathbf{X})$ of $y_{i}(\mathbf{X})$. Then
the PDD-MCS estimate of the failure probability $P_{F}(\mathbf{d})$
is
\begin{equation}
\tilde{P}_{F,PM}(\mathbf{d})=\mathbb{E}_{\mathbf{d}}\left[I_{\tilde{\Omega}_{F,S,m}}(\mathbf{X})\right]=\lim\limits _{L\rightarrow\infty}\frac{1}{L}\sum\limits _{l=1}^{L}I_{\tilde{\Omega}_{F,S,m}}(\mathbf{x}^{(l)}),\label{pfmcs}
\end{equation}
where the subscript "PM" stands for PDD-MCS, $L$ is the sample size, $\mathbf{x}^{(l)}$ is the $l$th realization
of $\mathbf{X}$, and $I_{\tilde{\Omega}_{F,S,m}}(\mathbf{x})$ is
another indicator function, which is equal to \emph{one} when $\mathbf{x}\in\tilde{\Omega}_{F,S,m}$
and \emph{zero} otherwise.

Note that the simulation of the PDD approximation in Equation \ref{pfmcs}
should not be confused with crude MCS commonly used for producing
benchmark results. The crude MCS, which requires numerical calculations
of $y(\mathbf{x}^{(l)})$ or $y_{i}(\mathbf{x}^{(l)})$ for input
samples $\mathbf{x}^{(l)},l=1,\cdots,L$, can be expensive or even
prohibitive, particularly when the sample size $L$ needs to be very
large for estimating small failure probabilities. In contrast, the
MCS embedded in PDD requires evaluations of simple analytical functions
that stem from an $S$-variate, $m$th-order approximation $\tilde{y}_{S,m}(\mathbf{x}^{(l)})$
or $\tilde{y}_{i,S,m}(\mathbf{x}^{(l)})$. Therefore, an arbitrarily
large sample size can be accommodated in the PDD-MCS method. In which
case, the PDD-MCS method also furnishes the approximate CDF $\tilde{F}_{y,PM}(\xi;\mathbf{d}):=P_{\mathbf{d}}[\tilde{y}_{S,m}(\mathbf{X})\le\xi]$
of $y(\mathbf{X})$ or even joint CDF of dependent stochastic responses,
if desired.

Although the PDD-SPA and PDD-MCS methods are both rooted in the same
PDD approximation, the former requires additional layers of approximations
to calculate the CGF and saddlepoint. Therefore, the PDD-SPA method,
when it works, is expected to be less accurate than the PDD-MCS method
at comparable computational efforts. However, the PDD-SPA method facilitates
an analytical means to estimate the probability distribution and reliability
$-$ a convenient process not supported by the PDD-MCS method. The
respective properties of both methods extend to sensitivity analysis,
presented in the following two sections.

\section{Design Sensitivity Analysis of Moments}

When solving RDO problems using gradient-based optimization algorithms,
at least first-order derivatives of the first and second moments of
a stochastic response with respect to each design variable are required.
In this section, a new method, developed by blending PDD with score
functions, for design sensitivity analysis of moments of an arbitrary
order, is presented.

\subsection{Score Functions}

Suppose that the first-order derivative of a moment $m^{(r)}(\mathbf{d})$,
where $r\in\mathbb{N}$, of a generic stochastic response $y(\mathbf{X})$
with respect to a design variable $d_{k}$, $1\le k\le M$, is sought.
Taking partial derivative of the moment with respect to $d_{k}$ and
then applying the Lebesgue dominated convergence theorem \cite{browder96},
which permits the differential and integral operators to be interchanged,
yields the sensitivity
\begin{equation}
\begin{array}{rcl}
{\displaystyle \frac{\partial m^{(r)}(\mathbf{d})}{\partial d_{k}}} & := & {\displaystyle \frac{\partial\mathbb{E}_{\mathbf{d}}\left[y^{r}(\mathbf{X})\right]}{\partial d_{k}}}\\
 & = & {\displaystyle \frac{\partial}{\partial d_{k}}}{\displaystyle \int_{\mathbb{R}^{N}}{\displaystyle y^{r}(\mathbf{x})f_{\mathbf{X}}(\mathbf{x};\mathbf{d})d\mathbf{x}}},\\
 & = & {\displaystyle \int_{\mathbb{R}^{N}}{\displaystyle y^{r}(\mathbf{x}){\displaystyle \frac{\partial\ln f_{\mathbf{X}}(\mathbf{x};\mathbf{d})}{\partial d_{k}}}f_{\mathbf{X}}(\mathbf{x};\mathbf{d})d\mathbf{x}}},\\
 & =: & {\displaystyle \mathbb{E}_{\mathbf{d}}\left[y^{r}(\mathbf{X})s_{d_{k}}^{(1)}(\mathbf{X};\mathbf{d})\right]}
\end{array}\label{momrsen}
\end{equation}
provided that $f_{\mathbf{X}}(\mathbf{x};\mathbf{d})>0$ and the derivative
$\partial\ln f_{\mathbf{X}}(\mathbf{x};\mathbf{d})\left/\partial d_{k}\right.$
exists. In last line of Equation \ref{momrsen},
$s_{d_{k}}^{(1)}(\mathbf{X};\mathbf{d}):=\partial\ln f_{\mathbf{X}}(\mathbf{X};\mathbf{d})\left/\partial d_{k}\right.$
is known as the first-order score function for the design variable
$d_{k}$ \cite{rubinstein93,rahman09}. In general, the sensitivities
are not available analytically since the moments are not either. Nonetheless,
the moments and their sensitivities have both been formulated as expectations
of stochastic quantities with respect to the same probability measure,
facilitating their concurrent evaluations in a single stochastic simulation
or analysis.
\begin{rem}
The evaluation of score functions, $s_{d_{k}}^{(1)}(\mathbf{X};\mathbf{d})$,
$k=1,\cdots,M$, requires differentiating only the PDF of $\mathbf{X}$.
Therefore, the resulting score functions can be determined easily
and, in many cases, analytically $-$ for instance, when $\mathbf{X}$
follows classical probability distributions \cite{rahman09}. If the
density function of $\mathbf{X}$ is arbitrarily prescribed, the score
functions can be calculated numerically, yet inexpensively, since
no evaluation of the performance function is involved.
\end{rem}
\,

When $\mathbf{X}$ comprises independent variables, as assumed here,
$\ln f_{\mathbf{X}}(\mathbf{X};\mathbf{d})=\sum_{i=1}^{i=N}\ln f_{X_{i}}(x_{i};\mathbf{d})$
is a sum of $N$ univariate log-density (marginal) functions of random
variables. Hence, in general, the score function for the $k$th design
variable, expressed by
\begin{equation}
s_{d_{k}}^{(1)}(\mathbf{X};\mathbf{d})={\displaystyle \sum_{i=1}^{N}}{\displaystyle {\displaystyle \frac{\partial\ln f_{X_{i}}(X_{i};\mathbf{d})}{\partial d_{k}}}}={\displaystyle \sum_{i=1}^{N}}s_{ki}(X_{i};\mathbf{d}),\label{sf}
\end{equation}
is also a sum of univariate functions $s_{ki}(X_{i};\mathbf{d}):=\partial\ln f_{X_{i}}(X_{i};\mathbf{d})\left/\partial d_{k}\right.$,
$i=1,\cdots,N$, which are the derivatives of log-density (marginal)
functions. If $d_{k}$ is a distribution parameter of a single random
variable $X_{i_{k}}$, then the score function reduces to $s_{d_{k}}^{(1)}(\mathbf{X};\mathbf{d})=\partial\ln f_{X_{i_{k}}}(X_{i_{k}};\mathbf{d})\left/\partial d_{k}\right.=:s_{ki_{k}}(X_{i_{k}};\mathbf{d})$,
the derivative of the log-density (marginal) function of $X_{i_{k}}$,
which remains a univariate function. Nonetheless, combining Equations
\ref{momrsen} and \ref{sf}, the sensitivity is obtained from
\begin{equation}
\frac{\partial m^{(r)}(\mathbf{d})}{\partial d_{k}}=\sum_{i=1}^{N}\mathbb{E}_{\mathbf{d}}\left[y^{r}(\mathbf{X})s_{ki}(X_{i};\mathbf{d})\right],\label{momrsen2}
\end{equation}
the sum of expectations of products comprising stochastic response
and log-density derivative functions with respect to the probability
measure $P_{\mathbf{d}}$, $\mathbf{d}\in\mathcal{D}$.

\subsection{Sensitivities of First- and Second-Order Moments}

For independent coordinates of $\mathbf{X}$, consider the Fourier-polynomial
expansion of the $k$th log-density derivative function
\begin{equation}
s_{ki}(X_{i};\mathbf{d})=s_{ki,\emptyset}(\mathbf{d})+{\displaystyle \sum_{j=1}^{\infty}}D_{k,ij}(\mathbf{d})\psi_{ij}(X_{i};\mathbf{d}),\label{sffpe}
\end{equation}
consisting of its own expansion coefficients
\begin{equation}
s_{ki,\emptyset}(\mathbf{d}):=\int_{\mathbb{R}}s_{ki}(x_{i};\mathbf{d})f_{X_{i}}(x_{i};\mathbf{d})dx_{i}\label{sfcoeff}
\end{equation}
and
\begin{equation}
D_{k,ij}(\mathbf{d}):=\int_{\mathbb{R}}s_{ki}(x_{i};\mathbf{d})\psi_{ij}(x_{i};\mathbf{d})f_{X_{i}}(x_{i};\mathbf{d})dx_{i}.\label{sfcoeff2}
\end{equation}
The expansion is valid if $s_{ki}$ is square integrable with respect
to the probability measure of $X_{i}$. When blended with the PDD
approximation, the score function leads to analytical or closed-form
expressions of the exact or approximate sensitivities as follows.

\subsubsection{Exact Sensitivities}

Employing Equations \ref{pddfull} and \ref{sffpe}, the product appearing on the right side of Equation
\ref{momrsen2} expands to
\begin{equation}
\begin{array}{rcl}
y^{r}(\mathbf{X})s_{ki}(X_{i};\mathbf{d}) & = & \left(y_{\emptyset}(\mathbf{d})+{\displaystyle \sum_{\emptyset\ne u\subseteq\{1,\cdots,N\}}}{\displaystyle \sum_{{\textstyle {\mathbf{j}_{|u|}\in\mathbb{N}_{0}^{|u|}\atop j_{1},\cdots,j_{|u|}\neq0}}}}\!\! C_{u\mathbf{j}_{|u|}}(\mathbf{d})\psi_{u\mathbf{j}_{|u|}}(\mathbf{X}_{u};\mathbf{d})\right)^{r}\times\\
 &  & \left(s_{ki,\emptyset}(\mathbf{d})+{\displaystyle \sum_{j=1}^{\infty}}D_{k,ij}(\mathbf{d})\psi_{ij}(X_{i};\mathbf{d})\right),
\end{array}\label{prod}
\end{equation}
encountering the same orthonormal polynomial bases that are consistent
with the probability measure $f_{\mathbf{X}}(\mathbf{x};\mathbf{d})d\mathbf{x}$.
The expectations of Equation \ref{prod} for $r=1$
and 2, aided by Propositions \ref{p1} and \ref{p2}, lead Equation
\ref{momrsen2} to
\begin{equation}
\frac{\partial m^{(1)}(\mathbf{d})}{\partial d_{k}}=\sum_{i=1}^{N}\left[y_{\emptyset}(\mathbf{d})s_{ki,\emptyset}(\mathbf{d})+{\displaystyle \sum_{j=1}^{\infty}}C_{ij}(\mathbf{d})D_{k,ij}(\mathbf{d})\right]\label{mom1sexact}
\end{equation}
and
\begin{equation}
\frac{\partial m^{(2)}(\mathbf{d})}{\partial d_{k}}=\sum_{i=1}^{N}\left[m^{(2)}(\mathbf{d})s_{ki,\emptyset}(\mathbf{d})+2y_{\emptyset}(\mathbf{d}){\displaystyle \sum_{j=1}^{\infty}}C_{ij}(\mathbf{d})D_{k,ij}(\mathbf{d})+T_{ki}\right],\label{mom2sexact}
\end{equation}
representing closed-form expressions of the sensitivities in terms
of the PDD or Fourier-polynomial expansion coefficients of the response
or log-density derivative functions. The last term on the right side of Equation \ref{mom2sexact} is
\begin{eqnarray}
T_{ki} & = & \sum_{i_{1}=1}^{N}{\displaystyle \sum_{i_{2}=1}^{N}}{\displaystyle \sum_{j_{1}=1}^{\infty}}{\displaystyle \sum_{j_{2}=1}^{\infty}}{\displaystyle \sum_{j_{3}=1}^{\infty}}C_{i_{1}j_{1}}(\mathbf{d})C_{i_{2}j_{2}}(\mathbf{d})D_{k,ij_{3}}(\mathbf{d})\times\nonumber \\
 &  & \mathbb{E_{\textrm{\textbf{d}}}}\left[\psi_{i_{1}j_{1}}(X_{i_{1}};\mathbf{d})\psi_{i_{2}j_{2}}(X_{i_{2}};\mathbf{d})\psi_{i_{3}j_{3}}(X_{i};\mathbf{d})\right],\label{tk}
\end{eqnarray}
which requires expectations of various products of three random orthonormal
polynomials and is further discussed in Subsection 3.2.4. Note that
these sensitivity equations are exact because PDD and Fourier-polynomial
expansions are exact representations of square-integrable functions.

\subsubsection{Approximate Sensitivities }

When $y(\mathbf{X})$ and $s_{ki}(X_{i};\mathbf{d})$ are replaced
by their $S$-variate, $m$th-order PDD and $m'$th-order Fourier-polynomial
approximations, respectively, the resultant sensitivity equations,
expressed by
\begin{equation}
\frac{\partial\tilde{m}_{S,m}^{(1)}(\mathbf{d})}{\partial d_{k}}:=\frac{\partial\mathbb{E}_{\mathbf{d}}\left[\tilde{y}_{S,m}(\mathbf{X})\right]}{\partial d_{k}}=\sum_{i=1}^{N}\left[y_{\emptyset}(\mathbf{d})s_{ki,\emptyset}(\mathbf{d})+{\displaystyle \sum_{j=1}^{m_{\min}}}C_{ij}(\mathbf{d})D_{k,ij}(\mathbf{d})\right]\label{mom1spdd}
\end{equation}
and
\begin{equation}
\frac{\partial\tilde{m}_{S,m}^{(2)}(\mathbf{d})}{\partial d_{k}}:=\frac{\partial\mathbb{E}_{\mathbf{d}}\left[\tilde{y}_{S,m}^{2}(\mathbf{X})\right]}{\partial d_{k}}=\sum_{i=1}^{N}\left[\tilde{m}_{S,m}^{(2)}(\mathbf{d})s_{ki,\emptyset}(\mathbf{d})+2y_{\emptyset}(\mathbf{d}){\displaystyle \sum_{j=1}^{m_{\min}}}C_{ij}(\mathbf{d})D_{k,ij}(\mathbf{d})+\tilde{T}_{ki,m,m'}\right],\label{mom2spdd}
\end{equation}
where $m_{\textrm{min}}:=\min(m,m')$ and
\begin{eqnarray}
\tilde{T}_{ki,m,m'} & = & \sum_{i_{1}=1}^{N}{\displaystyle \sum_{i_{2}=1}^{N}}{\displaystyle \sum_{j_{1}=1}^{m}}{\displaystyle \sum_{j_{2}=1}^{m}}{\displaystyle \sum_{j_{3}=1}^{m'}}C_{i_{1}j_{1}}(\mathbf{d})C_{i_{2}j_{2}}(\mathbf{d})D_{k,ij_{3}}(\mathbf{d})\times\nonumber \\
 &  & \mathbb{E_{\textrm{\textbf{d}}}}\left[\psi_{i_{1}j_{1}}(X_{i_{1}};\mathbf{d})\psi_{i_{2}j_{2}}(X_{i_{2}};\mathbf{d})\psi_{ij_{3}}(X_{i};\mathbf{d})\right],\label{tkmmprime}
\end{eqnarray}
become approximate, relying on the truncation parameters $S$, $m$,
and $m'$ in general. At appropriate limits, the approximate sensitivities
of the moments converge to exactness as described by Proposition \ref{p3}.
\begin{prop}
\label{p3}Let $\tilde{y}_{S,m}(\mathbf{X})$ be an $S$-variate,
$m$th-order PDD approximation of a square-integrable function $y(\mathbf{X})$,
where $\mathbf{X}=(X_{1},\cdots,X_{N})\in\mathbb{R}^{N}$ comprises
independent random variables with marginal probability distributions
$f_{X_{i}}(x_{i};\mathbf{d})$, $i=1,\cdots,N$, and $\mathbf{d}=(d_{1},\cdots,d_{M})\in\mathcal{D}$
is a design vector with non-empty closed set $\mathcal{D}\subseteq\mathbb{R}^{M}$.
Given the distribution parameter $d_{k}$, let the $k$th log-density
derivative function $s_{ki}(X_{i};\mathbf{d})$ of the $i$th random
variable $X_{i}$ be square integrable. Then for $k=1,\cdots M,$
\begin{equation}
\lim_{S\to N,\, m,m'\to\infty}\frac{\partial\tilde{m}_{S,m}^{(1)}(\mathbf{d})}{\partial d_{k}}=\frac{\partial m^{(1)}(\mathbf{d})}{\partial d_{k}}
\end{equation}
and
\begin{equation}
\lim_{S\to N,\, m,m'\to\infty}\frac{\partial\tilde{m}_{S,m}^{(2)}(\mathbf{d})}{\partial d_{k}}=\frac{\partial m^{(2)}(\mathbf{d})}{\partial d_{k}}.
\end{equation}
\end{prop}
\begin{proof}
Taking the limits $S\to N$, $m\to\infty$, and $m'\to\infty$ on
Equations \ref{mom1spdd} and \ref{mom2spdd} and recognizing $\tilde{m}_{S,m}^{(2)}(\mathbf{d})\to m^{(2)}(\mathbf{d})$
and $\tilde{T}_{ki,m,m'}\to T_{ki}$,
\begin{eqnarray}
\lim\limits _{S\rightarrow N,\ m,m'\rightarrow\infty}\frac{\partial\tilde{m}_{S,m}^{(1)}(\mathbf{d})}{\partial d_{k}} & = & \lim\limits _{S\rightarrow N,\ m,m'\rightarrow\infty}\;\sum_{i=1}^{N}\left[y_{\emptyset}(\mathbf{d})s_{ki,\emptyset}(\mathbf{d})+{\displaystyle \sum_{j=1}^{m_{\min}}}C_{ij}(\mathbf{d})D_{k,ij}(\mathbf{d})\right]\nonumber \\
 & = & \sum_{i=1}^{N}\left[y_{\emptyset}(\mathbf{d})s_{ki,\emptyset}(\mathbf{d})+{\displaystyle \sum_{j=1}^{\infty}}C_{ij}(\mathbf{d})D_{k,ij}(\mathbf{d})\right]\\
 & = & \frac{\partial m^{(1)}(\mathbf{d})}{\partial d_{k}}\nonumber
\end{eqnarray}
and
\begin{eqnarray}
 &  & \lim\limits _{s\rightarrow N,\ m,m'\rightarrow\infty}\frac{\partial\tilde{m}_{S,m}^{(2)}(\mathbf{d})}{\partial d_{k}}\nonumber \\
 & = & \lim\limits _{S\rightarrow N,\ m,m'\rightarrow\infty}\;\sum_{i=1}^{N}\left[\tilde{m}_{S,m}^{(2)^{2}}(\mathbf{d})s_{ki,\emptyset}(\mathbf{d})+2y_{\emptyset}(\mathbf{d}){\displaystyle \sum_{j=1}^{m_{\min}}}C_{ij}(\mathbf{d})D_{k,ij}(\mathbf{d})+\tilde{T}_{ki,m,m'}\right]\nonumber \\
 & = & \sum_{i=1}^{N}\left[m^{(2)^{2}}(\mathbf{d})s_{ki,\emptyset}(\mathbf{d})+2y_{\emptyset}(\mathbf{d}){\displaystyle \sum_{j=1}^{\infty}}C_{ij}(\mathbf{d})D_{k,ij}(\mathbf{d})+T_{ki}\right]\\
 & = & \frac{\partial m^{(2)}(\mathbf{d})}{\partial d_{k}},\nonumber
\end{eqnarray}
where the last lines follow from Equations \ref{mom1sexact} and \ref{mom2sexact}.
\end{proof}

Of the two sensitivities, $\partial\tilde{m}_{S,m}^{(1)}(\mathbf{d})/\partial d_{k}$
does not depend on $S$, meaning that both the univariate ($S=1$)
and bivariate ($S=2$) approximations, given the same $m_{\min}<\infty$,
form the same result, as displayed in Equation \ref{mom1spdd}. However, the sensitivity equations of $\partial\tilde{m}_{S,m}^{(2)}(\mathbf{d})/\partial d_{k}$
for the univariate and bivariate approximations vary with respect
to $S$, $m$, and $m'$. For instance, the univariate approximation
results in
\begin{equation}
\frac{\partial\tilde{m}_{1,m}^{(2)}(\mathbf{d})}{\partial d_{k}}=\sum_{i=1}^{N}\left[\tilde{m}_{1,m}^{(2)}(\mathbf{d})s_{ki,\emptyset}(\mathbf{d})+2y_{\emptyset}(\mathbf{d}){\displaystyle \sum_{j=1}^{m_{\min}}}C_{ij}(\mathbf{d})D_{k,ij}(\mathbf{d})+\tilde{T}_{ki,m,m'}\right],
\end{equation}
whereas the bivariate approximation yields
\begin{equation}
\frac{\partial\tilde{m}_{2,m}^{(2)}(\mathbf{d})}{\partial d_{k}}=\sum_{i=1}^{N}\left[\tilde{m}_{2,m}^{(2)}(\mathbf{d})s_{ki,\emptyset}(\mathbf{d})+2y_{\emptyset}(\mathbf{d}){\displaystyle \sum_{j=1}^{m_{\min}}}C_{ij}(\mathbf{d})D_{k,ij}(\mathbf{d})+\tilde{T}_{ki,m,m'}\right].
\end{equation}
Analogous to the moments, the univariate and bivariate approximations
of the sensitivities of the moments involve only univariate and at
most bivariate expansion coefficients of $y$, respectively. Since
the expansion coefficients of log-density derivative functions do
not involve the response function, no additional cost is incurred
from response analysis. In other words, the effort required to obtain
the statistical moments of a response also furnishes the sensitivities
of moments, a highly desirable trait for efficiently solving RDO problems.
\begin{rem}
Since the derivatives of log-density functions are univariate functions,
their expansion coefficients require only univariate integration for
their evaluations. When $X_{i}$ follows classical distributions $-$
for instance, the Gaussian distribution $-$ then the coefficients
can be calculated exactly or analytically. Otherwise, numerical quadrature
is required. Nonetheless, there is no need to employ dimension-reduction
integration for calculating the expansion coefficients of
the derivatives of log-density functions.
\end{rem}

\subsubsection{Special Cases}

There exist two special cases when the preceding expressions of the
sensitivities of moments simplify slightly. They are contingent on
how a distribution parameter affects the probability distributions
of random variables.

First, when $\mathbf{X}$ comprises independent variables such that
$d_{k}$ is a distribution parameter of a single random variable,
say, $X_{i_{k}}$, $1\le i_{k}\le N$ , then $s_{ki_{k}}(X_{i_{k}};\mathbf{d})$
$-$ the $k$th log-density derivative function of $X_{i_{k}}$ $-$
is the only relevant function of interest. Consequently, the expansion
coefficients $s_{ki,\emptyset}(\mathbf{d})=s_{ki_{k},\emptyset}(\mathbf{d})$
(say) and $D_{k,ij}(\mathbf{d})=D_{k,i_{k}j}(\mathbf{d})$ (say),
if $i=i_{k}$ and \emph{zero} otherwise. Moreover, the outer sums
of Equations \ref{mom1spdd} and \ref{mom2spdd} vanish, yielding
\begin{equation}
\frac{\partial\tilde{m}_{S,m}^{(1)}(\mathbf{d})}{\partial d_{k}}=y_{\emptyset}(\mathbf{d})s_{ki_{k},\emptyset}(\mathbf{d})+{\displaystyle \sum_{j=1}^{m_{\min}}}C_{i_{k}j}(\mathbf{d})D_{k,i_{k}j}(\mathbf{d})\label{mom1spdd2}
\end{equation}
and
\begin{equation}
\frac{\partial\tilde{m}_{S,m}^{(2)}(\mathbf{d})}{\partial d_{k}}=\tilde{m}_{S,m}^{(2)}(\mathbf{d})s_{ki_{k},\emptyset}(\mathbf{d})+2y_{\emptyset}(\mathbf{d}){\displaystyle \sum_{j=1}^{m_{\min}}}C_{i_{k}j}(\mathbf{d})D_{k,i_{k}j}(\mathbf{d})+\tilde{T}_{ki_{k},m,m'}.\label{mom2spdd2}
\end{equation}

Second, when $\mathbf{X}$ consists of independent and identical variables,
then $s_{ki}(X_{i};\mathbf{d})=s_{k}(X_{i};\mathbf{d})$ (say), that
is, the $k$th log-density derivative functions of all random variables
are alike. Accordingly, the expansion coefficients $s_{ki,\emptyset}(\mathbf{d})=s_{k,\emptyset}(\mathbf{d})$
(say) and $D_{k,ij}(\mathbf{d})=D_{k,j}(\mathbf{d})$ (say) for all
$i=1,\cdots,N$, producing
\begin{equation}
\frac{\partial\tilde{m}_{S,m}^{(1)}(\mathbf{d})}{\partial d_{k}}=\sum_{i=1}^{N}\left[y_{\emptyset}(\mathbf{d})s_{k,\emptyset}(\mathbf{d})+{\displaystyle \sum_{j=1}^{m_{\min}}}C_{ij}(\mathbf{d})D_{k,j}(\mathbf{d})\right]\label{mom1spdd3}
\end{equation}
and
\begin{equation}
\frac{\partial\tilde{m}_{S,m}^{(2)}(\mathbf{d})}{\partial d_{k}}=\sum_{i=1}^{N}\left[\tilde{m}_{S,m}^{(2)}(\mathbf{d})s_{k,\emptyset}(\mathbf{d})+2y_{\emptyset}(\mathbf{d}){\displaystyle \sum_{j=1}^{m_{\min}}}C_{ij}(\mathbf{d})D_{k,j}(\mathbf{d})+\tilde{T}_{ki,m,m'}\right].\label{mom2spdd3}
\end{equation}
It is important to clarify that the first special case, that is, Equations
\ref{mom1spdd2} and \ref{mom2spdd2}, coincide with those presented
in a previous work \cite{ren12} by the authors. However, the second
case, that is, Equations \ref{mom1spdd3} and \ref{mom2spdd3}, including
the generalized version, that is, Equations \ref{mom1spdd} and \ref{mom2spdd},
are new. The results of sensitivity equations from these two special
cases will be discussed in the Numerical Examples section.

\subsubsection{Evaluation of $\tilde{T}_{ki,m,m'}$}

The evaluation of $\tilde{T}_{ki,m,m'}$ in Equation \ref{tkmmprime}
requires expectations of various products of three random orthonormal
polynomials. The expectations vanish when $i_{1}\ne i_{2}\ne i_{3}$,
regardless of the probability measures of random variables. For classical
polynomials, such as Hermite, Laguerre, and Legendre polynomials,
there exist formulae for calculating the expectations when $i_{1}=i_{2}=i_{3}=i\;(\mathrm{say})$.

When $X_{i}$ follows the standard Gaussian distribution,
the expectations are determined from the properties of univariate
Hermite polynomials, yielding \cite{busbridge48}
\begin{equation}
\mathbb{E}_{\mathbf{d}}\left[\psi_{ij_{1}}(X_{i};\mathbf{d})\psi_{ij_{2}}(X_{i};\mathbf{d})\psi_{ij_{3}}(X_{i};\mathbf{d})\right]={\displaystyle \frac{\sqrt{j_{1}!j_{2}!j_{3}!}}{(q-j_{1})!(q-j_{2})!(q-j_{3})!}},
\end{equation}
if $q\in\mathbb{N}$, $2q=j_{1}+j_{2}+j_{3}$, and
$j_{1},j_{2},j_{3}\le q$, and \emph{zero} otherwise. When $X_{i}$ follows the exponential
distribution with unit mean, the expectations are attained from the
properties of univariate Laguerre polynomials, producing \cite{kleindienst93}
\begin{equation}
\begin{array}{rcl}
 &  & \mathbb{E}_{\mathbf{d}}\left[\psi_{ij_{1}}(X_{i};\mathbf{d})\psi_{ij_{2}}(X_{i};\mathbf{d})\psi_{ij_{3}}(X_{i};\mathbf{d})\right]\\
 & = & {\displaystyle (-1)^{j_{1}+j_{2}+j_{3}}}{\displaystyle {\displaystyle \sum\limits _{v=v_{\min}}^{v_{\max}}\frac{(j_{1}+j_{2}-v)!2^{j_{3}-j_{1}-j_{2}+2v}}{v!(j_{1}-v)!(j_{2}-v)!}}}{\displaystyle \binom{v}{j_{3}-j_{1}-j_{2}+2v}},
\end{array}
\end{equation}
if $|j_{1}-j_{2}|\le j_{3}\le j_{1}+j_{2}$, and
\emph{zero} otherwise, where
$v_{\min}=\frac{1}{2}(j_{1}+j_{2}+1-j_{3})$, $v_{\max}=\min(j_{1},j_{2},j_{1}+j_{2}-j_{3})$.
When $X_{i}$ follows the uniform distribution on the interval $[-1,1]$,
the expectations are obtained from the properties
of univariate Legendre polynomials, forming \cite{kleindienst93}
\begin{equation}
\begin{array}{rcl}
 &  & \mathbb{E}_{\mathbf{d}}\left[\psi_{ij_{1}}(X_{i};\mathbf{d})\psi_{ij_{2}}(X_{i};\mathbf{d})\psi_{ij_{3}}(X_{i};\mathbf{d})\right]\\
 & = & {\displaystyle \frac{1}{2}\sqrt{2(2j_{1}+1)(2j_{2}+1)(2j_{3}+1)}\times}\\
 &  & {\displaystyle \frac{(j_{1}+j_{2}-j_{3}-1)!!(j_{2}+j_{3}-j_{1}-1)!!(j_{1}+j_{2}+j_{3})!!(j_{1}+j_{3}-j_{2}-1)!!}{(j_{1}+j_{2}-j_{3})!!(j_{2}+j_{3}-j_{1})!!(j_{1}+j_{2}+j_{3}+1)!!(j_{1}+j_{3}-j_{2})!!},}
\end{array}\label{3termUniform}
\end{equation}
if $q\in\mathbb{N}$, $2q=j_{1}+j_{2}+j_{3}$, and
$|j_{1}-j_{2}|\le j_{3}\le j_{1}+j_{2}$, and \emph{zero} otherwise. The symbol $!!$ in Equation \ref{3termUniform} denotes
the double factorial. However, deriving a master formula for arbitrary
probability distributions of $X_{i}$ is impossible. In which case,
the non-trivial solution of the expectation can be obtained by numerical
integration of elementary functions.

\subsection{Sensitivities of Higher-Order Moments}

No closed-form or analytical expressions are possible for calculating
sensitivities of higher-order ($2<r<\infty)$ moments by the PDD approximation.
Two options, consistent with statistical moment analysis in Subsection
2.2, are proposed for sensitivity analysis.

In Option I, the sensitivity is obtained by replacing $y$ by $\tilde{y}_{S,m}$
in Equation \ref{momrsen} and utilizing Equations \ref{yrPDD} and
\ref{sf}, resulting in
\begin{equation}
\begin{array}{rcl}
{\displaystyle \frac{\partial\tilde{m}_{S,m}^{(r)}(\mathbf{d})}{\partial d_{k}}} & = & \int_{\mathbb{R}^{N}}\tilde{y}_{S,m}^{r}(\mathbf{x})s_{d_{k}}^{(1)}(\mathbf{x};\mathbf{d})f_{\mathbf{\mathbf{X}}}(\mathbf{x};\mathbf{d})d\mathbf{x}\\
 & = & g_{\emptyset}(\mathbf{d}){\displaystyle \sum_{i=1}^{N}}\int_{\mathbb{R}}s_{ki}(x_{i};\mathbf{d})f_{X_{i}}(x_{i};\mathbf{d})dx_{i}+\\
 &  & {\displaystyle \sum_{i=1}^{N}}{\displaystyle \sum_{{\textstyle {\emptyset\ne u\subseteq\{1,\cdots,N\},\, i\in u\atop 1\le|u|\le\min(rS,N)}}}\int_{\mathbb{R}^{|u|}}g_{u}(\mathbf{x}_{u};\mathbf{d})s_{ki}(x_{i};\mathbf{d})f_{\mathbf{\mathbf{X}}_{u}}(\mathbf{x}_{u};\mathbf{d})d\mathbf{x}_{u}}+\\
 &  & {\displaystyle \sum_{i=1}^{N}}{\displaystyle \sum_{{\textstyle {\emptyset\ne u\subseteq\{1,\cdots,N\},\, i\notin u\atop 1\le|u|\le\min(rS,N)}}}\int_{\mathbb{R}^{|u|}}g_{u}(\mathbf{x}_{u};\mathbf{d})f_{\mathbf{\mathbf{X}}_{u}}(\mathbf{x}_{u};\mathbf{d})d\mathbf{x}_{u}}\times\\
 &  & \int_{\mathbb{R}}s_{ki}(x_{i};\mathbf{d})f_{X_{i}}(x_{i};\mathbf{d})dx_{i},
\end{array}\label{momrsen3}
\end{equation}
which involves at most $\min(rS,N)$-dimensional integrations. Similar
to statistical moment analysis, this option becomes impractical when
$\min(rS,N)$ is large or numerous $\min(rS,N)$-dimensional
integrations are required.

In contrast, the sensitivity in Option II is attained by replacing
$\tilde{y}_{S,m}^{r}$ by $\tilde{z}_{\bar{S},\bar{m}}$ in the first
line of Equation \ref{momrsen3}, yielding
\begin{equation}
\begin{array}{rcl}
 &  & {\displaystyle \frac{\partial\tilde{m}_{S,m}^{(r)}(\mathbf{d})}{\partial d_{k}}}\\
 & \cong & \int_{\mathbb{R}^{N}}\tilde{z}_{\bar{S},\bar{m}}(\mathbf{x})s_{d_{k}}^{(1)}(\mathbf{x};\mathbf{d})f_{\mathbf{\mathbf{X}}}(\mathbf{x};\mathbf{d})d\mathbf{x}\\
 & = & z_{\emptyset}(\mathbf{d}){\displaystyle \sum_{i=1}^{N}}\int_{\mathbb{R}}s_{ki}(x_{i};\mathbf{d})f_{X_{i}}(x_{i};\mathbf{d})dx_{i}+\\
 &  & {\displaystyle \sum_{i=1}^{N}}{\displaystyle \sum_{{\textstyle {\emptyset\ne u\subseteq\{1,\cdots,N\}\atop 1\le|u|\le\bar{S},\, i\in u}}}}{\displaystyle \sum_{{\textstyle {\mathbf{j}_{|u|}\in\mathbb{N}_{0}^{|u|},||\mathbf{j}_{|u|}||_{\infty}\le\bar{m}\atop j_{1},\cdots,j_{|u|}\neq0}}}}\bar{C}_{u\mathbf{j}_{|u|}}(\mathbf{d})\int_{\mathbb{R}^{|u|}}\psi_{u\mathbf{j}_{|u|}}(\mathbf{x}_{u};\mathbf{d})s_{ki}(x_{i};\mathbf{d})f_{\mathbf{\mathbf{X}}_{u}}(\mathbf{x}_{u};\mathbf{d})d\mathbf{x}_{u},
\end{array}\label{momrsen4}
\end{equation}
requiring at most $\bar{S}$-dimensional integrations of at most
$\bar{m}$th-order polynomials, where the terms
related to $i\notin u$ vanish as per Proposition \ref{p1}. Therefore, a significant gain in efficiency
is possible in Option II for sensitivity analysis as well. The sensitivity
equations further simplify for special cases, as explained in Section
3.2. Nonetheless, numerical integrations are necessary for calculating
the sensitivities by either option.

\section{Design Sensitivity Analysis of Reliability}

When solving RBDO problems using gradient-based optimization algorithms,
at least first-order derivatives of the failure probability with respect
to each design variable is required. Two methods for the sensitivity
analysis of the failure probability, named the PDD-SPA and PDD-MCS
methods, are presented.

\subsection{The PDD-SPA Method}

Suppose that the first-order derivative $\partial\tilde{F}_{y,PS}(\xi;\mathbf{d})/\partial d_{k}$
of the CDF $\tilde{F}_{y,PS}(\xi;\mathbf{d})$ of $\tilde{y}_{S,m}(\mathbf{X})$,
obtained by the PDD-SPA method, with respect to a design variable
$d_{k}$, is desired. Applying the chain rule on the derivative of
Equation \ref{lr},
\begin{equation}
\frac{\partial\tilde{F}_{y,PS}(\xi;\mathbf{d})}{\partial d_{k}}=\sum\limits _{r=1}^{Q}\left(\frac{\partial\tilde{F}_{y,PS}}{\partial w}\frac{\partial w}{\partial\tilde{\kappa}_{S,m}^{(r)}}+\frac{\partial\tilde{F}_{y,PS}}{\partial v}\frac{\partial v}{\partial\tilde{\kappa}_{S,m}^{(r)}}\right)\frac{\partial\tilde{\kappa}_{S,m}^{(r)}}{\partial d_{k}}\label{pdspdd}
\end{equation}
is obtained via the partial derivatives
\begin{equation}
\frac{\partial\tilde{F}_{y,PS}}{\partial w}=\phi(w)\left(\frac{w}{v}-\frac{1}{w^{2}}\right),\;\frac{\partial\tilde{F}_{y,PS}}{\partial v}=\frac{\phi(w)}{v^{2}},\label{pdspdd2}
\end{equation}
\begin{eqnarray}
\frac{\partial\tilde{\kappa}_{S,m}^{(r)}}{\partial d_{k}}=\left\{ \begin{array}{l@{\quad:\quad}l}
{\displaystyle \frac{\partial\tilde{m}_{S,m}^{(1)}(\mathbf{d})}{\partial d_{k}}} & r=1,\\
{\displaystyle \frac{\partial\tilde{m}_{S,m}^{(r)}(\mathbf{d})}{\partial d_{k}}-\sum\limits _{p=1}^{r-1}\binom{r-1}{p-1}\left(\frac{\partial\tilde{\kappa}_{S,m}^{(r)}}{\partial d_{k}}\tilde{m}_{S,m}^{(r-p)}(\mathbf{d})+\tilde{\kappa}_{S,m}^{(p)}\frac{\partial\tilde{m}_{S,m}^{(r-p)}}{\partial d_{k}}\right)} & 2\le r\le Q,
\end{array}\right.\label{pdspdd3}
\end{eqnarray}
where the derivatives of moments, that is, $\partial\tilde{m}_{S,m}^{(r)}/\partial d_{k}$,
$r=1,\cdots,Q$, required to calculate the derivatives of cumulants,
are obtained using score functions, as described in Section 3. The
remaining two partial derivatives are expressed by
\begin{equation}
\frac{\partial w}{\partial\tilde{\kappa}_{S,m}^{(r)}}=\frac{\partial w}{\partial t_{s}}\frac{\partial t_{s}}{\partial\tilde{\kappa}_{S,m}^{(r)}}+\frac{\partial w}{\partial\tilde{K}_{y,Q,S,m}}\left[\frac{\partial\tilde{K}_{y,Q,S,m}}{\partial\tilde{\kappa}_{S,m}^{(r)}}+\frac{\partial\tilde{K}_{y,Q,S,m}}{\partial t_{s}}\frac{\partial t_{s}}{\partial\tilde{\kappa}_{S,m}^{(r)}}\right],\label{pdspdd4}
\end{equation}
and
\begin{equation}
\frac{\partial v}{\partial\tilde{\kappa}_{S,m}^{(r)}}=\frac{\partial v}{\partial t_{s}}\frac{\partial t_{s}}{\partial\tilde{\kappa}_{S,m}^{(r)}}+\frac{\partial v}{\partial\tilde{K}''_{y,Q,S,m}}\left[\frac{\partial\tilde{K}''_{y,Q,S,m}}{\partial\tilde{\kappa}_{S,m}^{(r)}}+\frac{\partial\tilde{K}''_{y,Q,S,m}}{\partial t_{s}}\frac{\partial t_{s}}{\partial\tilde{\kappa}_{S,m}^{(r)}}\right],\label{pdspdd5}
\end{equation}
where
\begin{equation}
\frac{\partial w}{\partial t_{s}}=\frac{\xi}{w},\;\frac{\partial w}{\partial\tilde{K}_{y,Q,S,m}}=-\frac{1}{w},\;\frac{\partial\tilde{K}_{y,Q,S,m}}{\partial t_{s}}=\xi,\;\frac{\partial v}{\partial t_{s}}=\left[\tilde{K}_{y,Q,S,m}''\right]^{\frac{1}{2}},\label{pdspdd6}
\end{equation}
\begin{equation}
\frac{\partial v}{\partial\tilde{K}''_{y,Q,S,m}}=\frac{t_{s}}{2\sqrt{\tilde{K}_{y,Q,S,m}''}},\;\frac{\partial t_{s}}{\partial\tilde{\kappa}_{S,m}^{(r)}}=-{\displaystyle \frac{{\displaystyle \frac{\partial\tilde{K}'_{y,Q,S,m}}{\partial\tilde{\kappa}_{S,m}^{(r)}}}}{{\displaystyle \frac{\partial\tilde{K}'_{y,Q,S,m}}{\partial t_{s}}}}.} \label{pdspdd7}
\end{equation}
The expressions of the partial derivatives $\partial\tilde{K}_{y,Q,S,m}/\partial\tilde{\kappa}_{S,m}^{(r)}$,
$\partial\tilde{K}'_{y,Q,S,m}/\partial\tilde{\kappa}_{S,m}^{(r)}$,
and $\partial\tilde{K}''_{y,Q,S,m}/\partial\tilde{\kappa}_{S,m}^{(r)}$,
not explicitly presented here, can be easily derived from Equation
\ref{cgf3} once the cumulants $\tilde{\kappa}_{S,m}^{(r)},\ r=1,\cdots,Q$,
and the saddlepoint $t_{s}$ are obtained. Similar sensitivity equations
were reported by Huang and Zhang \cite{huang12}. However, Equation
\ref{pdspdd} is built on the PDD approximation of a stochastic response,
as opposed to the RDD approximation used by Huang and Zhang. Furthermore,
no transformations of random variables are necessary in the proposed
PDD-SPA method.

Henceforth, the first-order derivative of the failure probability
estimate by the PDD-SPA method is easily determined from
\begin{equation}
\frac{\partial\tilde{P}_{F,PS}(\mathbf{d})}{\partial d_{k}}=\frac{\partial\tilde{F}_{y,PS}(0;\mathbf{d})}{\partial d_{k}},
\end{equation}
the sensitivity of the probability distribution evaluated at $\xi=0$.
Algorithm 1 describes the procedure of the PDD-SPA method for calculating
the reliability and its design sensitivity of a general stochastic
problem.

\begin{algorithm}
\caption{Numerical implementation of the PDD-SPA method for CDF $\tilde{F}_{y,PS}(\xi;\mathbf{d})$
and its sensitivity $\partial{\tilde{F}_{y,PS}(\xi;\mathbf{d})}/\partial{d_{k}}$
\label{alg:PDD-SPA}}

  \begin{algorithmic}[0]    \scriptsize     \State Define $\xi$ and $\mathbf{d}$      \State Specify $S$, $\bar{S}$, $m$, $\bar{m}$, and $Q$   \State Obtain the PDD approximation $\tilde{y}_{S,m}(\mathbf{X})$ \Comment{[from Equation \ref{pdd}]}   \For{$r \gets 1 \textrm{ to } Q$}       \State Calculate $\tilde{m}^{(r)}_{S,m}(\mathbf{d})$ \Comment{\begin{minipage}[t]{2.9in} [from Equation \ref{momr} for Option I, or Equation \ref{momr2} for Option II; \par\hspace{0.5mm} if $r=1$ and $2$, then Equations \ref{mom1} and \ref{mom2} can be used] \end{minipage}}        \State Calculate  $\partial{\tilde{m}^{(r)}_{S,m}(\mathbf{d})}/\partial{d_{k}} $ \Comment{\begin{minipage}[t]{2.9in} [from Equation \ref{momrsen3} for Option I, or Equation \ref{momrsen4} for Option II; \par\hspace{0.5mm} if $r=1$ and $2$, then Equations \ref{mom1spdd} and \ref{mom2spdd} can be used] \end{minipage}}        \EndFor       \For{$r \gets 1 \textrm{ to } Q$}       \State Calculate $\tilde{\kappa}^{(r)}_{S,m}(\mathbf{d})$ \Comment{[from Equation \ref{mom2cum}]}       \State Calculate  $\partial{\tilde{\kappa}^{(r)}_{S,m}(\mathbf{d})}/\partial{d_{k}} $ \Comment{[from Equation \ref{pdspdd3}]}       \EndFor      \State Obtain interval $(t_{l}, t_{u})$ for the saddlepoint    \Comment{[from Table \ref{tab:1-8case} if $Q=4$]}     \State Calculate $\tilde{K}'_{y,Q,S,m}(t_{l};\mathbf{d})$ and $\tilde{K}'_{y,Q,S,m}(t_{u};\mathbf{d})$ \Comment{[from Equation \ref{cgf3}]}     \If{$\xi\in(\tilde{K}'_{y,Q,S,m}(t_{l};\mathbf{d}),\tilde{K}'_{y,Q,S,m}(t_{u};\mathbf{d}))$ }  \State Calculate saddlepoint $t_{s}$   \Comment{[from Equations \ref{cgf3} and \ref{sp}]}       \State Calculate $\tilde{F}_{y,PS}(\xi;\mathbf{d})$  \Comment{[from Equation \ref{lr}]}     \State Calculate $\partial{\tilde{F}_{y,PS}(\xi;\mathbf{d})}/\partial{d_{k}}$  \Comment{[from Equations \ref{pdspdd}-\ref{pdspdd7}]}    \Else   \State Stop    \Comment{[the PDD-SPA method fails]}    \EndIf   \end{algorithmic}
\end{algorithm}

\subsection{The PDD-MCS Method}

Taking a partial derivative of the PDD-MCS estimate of the failure
probability in Equation \ref{pfmcs} with respect to $d_{k}$ and
then following the same arguments in deriving Equation \ref{momrsen}
produces
\begin{equation}
\begin{array}{rcl}
{\displaystyle \frac{\partial\tilde{P}_{F,PM}(\mathbf{d})}{\partial d_{k}}} & := & {\displaystyle \frac{\partial\mathbb{E}_{\mathbf{d}}\left[I_{\tilde{\Omega}_{F,S,m}}(\mathbf{X})\right]}{\partial d_{k}}}\\
 & = & \mathbb{E}_{\mathbf{d}}\left[I_{\tilde{\Omega}_{F,S,m}}(\mathbf{X})s_{d_{k}}^{(1)}(\mathbf{X};\mathbf{d})\right]\\
 & = & {\displaystyle \lim_{L\rightarrow\infty}}{\displaystyle \frac{1}{L}\sum_{l=1}^{L}}\left[I_{\tilde{\Omega}_{F,S,m}}(\mathbf{x}^{(l)})s_{d_{k}}^{(1)}(\mathbf{x}^{(l)};\mathbf{d})\right],
\end{array}\label{pfsmcs}
\end{equation}
where $L$ is the sample size, $\mathbf{x}^{(l)}$ is the $l$th realization
of $\mathbf{X}$, and $I_{\tilde{\Omega}_{F,S,m}}(\mathbf{x})$ is
the PDD-generated indicator function, which is equal to \emph{one}
when $\mathbf{x}\in\tilde{\Omega}_{F,S,m}$ and \emph{zero} otherwise.
Again, they are easily and inexpensively determined by sampling analytical
functions that describe $\tilde{y}_{S,m}$ and $s_{d_{k}}^{(1)}$.
A similar sampling procedure can be employed to
calculate the sensitivity of the PDD-MCS generated CDF $\tilde{F}_{y,PM}(\xi;\mathbf{d}):=P_{\mathbf{d}}[\tilde{y}_{S,m}(\mathbf{X})\le\xi]$.
It is important to note that the effort required to calculate the
failure probability or CDF also delivers their sensitivities, incurring
no additional cost. Setting $S=1$ or $2$ in Equations \ref{pfmcs}
and \ref{pfsmcs}, the univariate or bivariate approximation of the
failure probability and its sensitivities are determined.
\begin{rem}
It is important to recognize that no Fourier-polynomial expansions
of the derivatives of log-density functions are required or invoked
in the PDD-MCS method for sensitivity analysis of failure probability.
This is in contrast to the sensitivity analysis of the first two moments,
where such Fourier-polynomial expansions aid in generating analytical
expressions of the sensitivities. No analytical expressions are possible
in the PDD-MCS method for sensitivity analysis of reliability or probability
distribution of a general stochastic response.
\end{rem}
\,
\begin{rem}
The score function method has the nice property that it requires differentiating
only the underlying PDF $f_{\mathbf{X}}(\mathbf{x};\mathbf{d})$.
The resulting score functions can be easily and, in most cases, analytically
determined. If the performance function is not differentiable or discontinuous
$-$ for example, the indicator function that comes from reliability
analysis $-$ the proposed method still allows evaluation of the sensitivity
if the density function is differentiable. In reality, the density
function is often smoother than the performance function, and therefore
the proposed sensitivity methods will be able to calculate sensitivities
for a wide variety of complex mechanical systems.
\end{rem}

\section{Calculation of Expansion Coefficients}

The determination of PDD expansion coefficients $y_{\emptyset}(\mathbf{d})$
and $C_{u\mathbf{j}_{|u|}}(\mathbf{d})$, where $\emptyset\ne u\subseteq\{1,\cdots,N\}$
and $\mathbf{j}_{|u|}\in\mathbb{N}_{0}^{|u|}|$; $||\mathbf{j}_{|u|}||_{\infty}\le m$;
$j_{1},\cdots,j_{|u|}\neq0$, is vitally important for evaluating
the statistical moments and probabilistic characteristics, including
their design sensitivities, of stochastic responses. The coefficients,
defined in Equations \ref{ANOVA2} and \ref{coeff}, involve various
$N$-dimensional integrals over $\mathbb{R}^{N}$. For large $N$,
a full numerical integration employing an $N$-dimensional tensor
product of a univariate quadrature formula is computationally prohibitive
and is, therefore, ruled out. The authors propose that the dimension-reduction
integration scheme, developed by Xu and Rahman \cite{xu04}, followed
by numerical quadrature, be used to estimate the coefficients accurately
and efficiently.

\subsection{Dimension-Reduction Integration}

Let $\mathbf{c}=(c_{1},\cdots,c_{N})\in\mathbb{R}^{N}$, which is
commonly adopted as the mean of $\mathbf{X}$, be a reference point,
and $y(\mathbf{x}_{v},\mathbf{c}_{-v})$ represent an $|v|$-variate
RDD component function of $y(\mathbf{x})$, where $v\subseteq\{1,\cdots,N\}$
\cite{rahman11,rahman12}. Given a positive integer $S\le R\le N$,
when $y(\mathbf{x})$ in Equations \ref{ANOVA2} and \ref{coeff}
is replaced with its $R$-variate RDD approximation, the coefficients $y_{\emptyset}(\mathbf{d})$ and $C_{u\mathbf{j}_{|u|}}(\mathbf{d})$
are estimated from \cite{xu04}
\begin{equation}
y_{\emptyset}(\mathbf{d})\cong{\displaystyle \sum_{i=0}^{R}}(-1)^{i}{N-R+i-1 \choose i}\sum_{{\textstyle {v\subseteq\{1,\cdots,N\}\atop |v|=R-i}}}\!\int_{\mathbb{R}^{|v|}}{\displaystyle y(\mathbf{x}_{v},\mathbf{c}_{-v})}f_{\mathbf{X}_{v}}(\mathbf{x}_{v};\mathbf{d})d\mathbf{x}_{v}\label{pddcoeff1}
\end{equation}
 and
\begin{equation}
C_{u\mathbf{j}_{|u|}}(\mathbf{d})\cong{\displaystyle \sum_{i=0}^{R}}(-1)^{i}{N-R+i-1 \choose i}\sum_{{\textstyle {v\subseteq\{1,\cdots,N\}\atop |v|=R-i,u\subseteq v}}}\!\int_{\mathbb{R}^{|v|}}{\displaystyle y(\mathbf{x}_{v},\mathbf{c}_{-v})\psi_{u\mathbf{j}_{|u|}}(\mathbf{x}_{u})}f_{\mathbf{X}_{v}}(\mathbf{x}_{v};\mathbf{d})d\mathbf{x}_{v},\label{pddcoeff2}
\end{equation}
respectively, requiring evaluation of at most $R$-dimensional integrals.
The reduced integration facilitates calculation of the coefficients
approaching their exact values as $R\to N$, and is significantly
more efficient than performing one $N$-dimensional integration, particularly
when $R\ll N$. Hence, the computational effort is significantly lowered
using the dimension-reduction integration. For instance, when $R=1$
or $2$, Equations \ref{pddcoeff1} and \ref{pddcoeff2} involve one-,
or at most, two-dimensional integrations, respectively.

For a general function $y$, numerical integrations are still required
for performing various $|v|$-dimensional integrals over $\mathbb{R}^{|v|}$,
$0\le|v|\le R$, in Equations \ref{pddcoeff1} and \ref{pddcoeff2}.
When $R>1$, multivariate numerical integrations are conducted by
constructing a tensor product of underlying univariate quadrature
rules. For a given $v\subseteq\{1,\cdots,N\}$, $1<|v|\le R$, let
$v=\{i_{1},\cdots i_{|v|}\}$, where $1\le i_{1}<\cdots<i_{|v|}\le N$.
Denote by $\{x_{i_{p}}^{(1)},\cdots,x_{i_{p}}^{(n)}\}\subset\mathbb{R}$
a set of integration points of $x_{i_{p}}$ and by $\{w_{i_{p}}^{(1)},\cdots,w_{i_{p}}^{(n)}\}$
the associated weights generated from a chosen univariate quadrature
rule and a positive integer $n\in\mathbb{N}$. Denote by $P^{(n)}=\times_{p=1}^{p=|v|}\{x_{i_{p}}^{(1)},\cdots,x_{i_{p}}^{(n)}\}$
a rectangular grid consisting of all integration points generated
by the variables indexed by the elements of $v$. Then the coefficients
using dimension-reduction integration and numerical quadrature are
approximated by
\begin{equation}
y_{\emptyset}(\mathbf{d})\cong{\displaystyle \sum_{i=0}^{R}}(-1)^{i}{N-R+i-1 \choose i}\sum_{{\textstyle {v\subseteq\{1,\cdots,N\}\atop |v|=R-i}}}\!\sum_{\mathbf{k}_{|v|}\in P^{(n)}}w^{(\mathbf{k}_{|v|})}y(\mathbf{x}_{v}^{(\mathbf{k}_{|v|})},\mathbf{c}_{-v})\label{pddcoeff1b}
\end{equation}
and
\begin{equation}
C_{u\mathbf{j}_{|u|}}(\mathbf{d})\cong{\displaystyle \sum_{i=0}^{R}}(-1)^{i}{N-R+i-1 \choose i}\sum_{{\textstyle {v\subseteq\{1,\cdots,N\}\atop |v|=R-i,u\subseteq v}}}\!\sum_{\mathbf{k}_{|v|}\in P^{(n)}}w^{(\mathbf{k}_{|v|})}y(\mathbf{x}_{v}^{(\mathbf{k}_{|v|})},\mathbf{c}_{-v})\psi_{u\mathbf{j}_{|u|}}(\mathbf{x}_{u}^{(\mathbf{k}_{|u|})}),\label{pddcoeff2b}
\end{equation}
where $\mathbf{x}_{v}^{(\mathbf{k}_{|v|})}=\{x_{i_{1}}^{(k_{1})},\cdots,x_{i_{|v|}}^{(k_{|v|})}\}$
and $w^{(\mathbf{k}_{|v|})}=\prod_{p=1}^{p=|v|}w_{i_{p}}^{(k_{p})}$
is the product of integration weights generated by the variables indexed
by the elements of $v$. Similarly, the coefficients $z_{\emptyset}(\mathbf{d})$
and $\bar{C}_{u\mathbf{j}_{|u|}}(\mathbf{d})$ of an $\bar{S}$-variate,
$\bar{m}$th-order PDD approximation of $\tilde{y}_{S,m}^{r}(\mathbf{X})$,
required in Option II for obtaining higher-order moments and their
sensitivities, can also be estimated from the dimension-reduction
integration. For independent coordinates of $\mathbf{X}$, as assumed
here, a univariate Gauss quadrature rule is commonly used, where the
integration points and associated weights depend on the probability
distribution of $X_{i}$. They are readily available, for example,
the Gauss-Hermite or Gauss-Legendre quadrature rule, when $X_{i}$
follows Gaussian or uniform distribution. For an arbitrary probability
distribution of $X_{i}$, the Stieltjes procedure can be employed
to generate the measure-consistent Gauss quadrature formulae \cite{rahman09b,gautschi04}.
An $n$-point Gauss quadrature rule exactly integrates a polynomial
with a total degree of at most $2n-1$.

\subsection{Computational Expense}

The $S$-variate, $m$th-order PDD approximation
requires evaluations of $\sum_{k=0}^{k=S}\binom{N}{k}m^{k}$ expansion
coefficients, including $y_{\emptyset}(\mathbf{d})$. If these coefficients
are estimated by dimension-reduction integration with $R=S<N$ and,
therefore, involve at most an $S$-dimensional tensor product of an
$n$-point univariate quadrature rule depending on $m$, then the
total cost for the $S$-variate, $m$th-order approximation entails
a maximum of $\sum_{k=0}^{k=S}\binom{N}{k}n^{k}(m)$ function evaluations.
If the integration points include a common point in each coordinate
$-$ a special case of symmetric input probability density functions
and odd values of $n$ $-$ the number of function evaluations reduces
to $\sum_{k=0}^{k=S}\binom{N}{k}(n(m)-1)^{k}$. Nonetheless, the computational
complexity of the $S$-variate PDD approximation is an $S$th-order
polynomial with respect to the number of random variables or integration
points. Therefore, PDD with dimension-reduction integration of the
expansion coefficients alleviates the curse of dimensionality to an
extent determined by $S$.

\section{Numerical Examples}

Six numerical examples, comprising various mathematical functions
and solid-mechanics problems, are illustrated to examine the accuracy,
efficiency, and convergence properties of the PDD methods developed
for calculating the first-order sensitivities of statistical moments,
probability distributions, and reliability. The PDD expansion coefficients
were estimated by dimension-reduction integration with the mean input
as the reference point, $R=S$, and $n=m+1$, where $S$ and $m$
vary depending on the problem. In all examples, orthonormal polynomials
and associated Gauss quadrature rules consistent with the probability
distributions of input variables, including classical forms, if they
exist, were employed. The first three examples entail independent
and identical random variables, where $d_{k}$ is a distribution parameter
of all random variables, whereas the last three examples contain merely
independent random variables, where $d_{k}$ is a distribution parameter
of a single random variable. The sample size for the embedded simulation
of the PDD-MCS method is $10^{6}$ in Examples 2 and 3, and $10^{7}$
in Example 5. Whenever possible, the exact sensitivities were applied
to verify the proposed methods, as in Examples 1 and 3. However, in
Examples 2, 4, and 5, which do not support exact solutions, the benchmark
results were generated from at least one of two
crude MCS-based approaches: (1) crude MCS in conjunction with score
functions (crude MCS/SF), which requires sampling of both the original
function $y$ and the score function $s_{d_{k}}^{(1)}$; and (2) crude
MCS in tandem with one-percent perturbation of finite-difference analysis
(crude MCS/FD), which entails sampling of the original function $y$
only. The sample size for either version of the crude MCS is $10^{6}$
in Examples 2, 3, and 4, and $10^{7}$ in Example 5. The derivatives
of log-density functions associated with the five types of random
variables used in all examples are described in Table \ref{tab:2-score}.

\begin{table}[tbph]
\centering
\begin{threeparttable}
\caption{Derivatives of log-density functions for various probability distributions}
\label{tab:2-score}
\begin{centering}
\renewcommand*\arraystretch{1.5}
\begin{tabular}{>{\centering}p{0.58in}>{\centering}p{0.3in}>{\centering}p{1.3in}>{\centering}p{1.2in}>{\centering}p{1.3in}}
\hline
{\footnotesize Distribution}  & $\mathbf{d}$ & {\footnotesize $f_{X}(x;\mathbf{d})$}  & {\footnotesize $\frac{\partial{\ln f_{X}(x;\mathbf{d})}}{\partial{d_{1}}}$}  & {\footnotesize $\frac{\partial{\ln f_{X}(x;\mathbf{d})}}{\partial{d_{2}}}$} \tabularnewline
\hline
{\footnotesize Exponential}  &{\footnotesize$\{\lambda\}^{{\scriptsize T}}$} & {\footnotesize $\lambda\exp\left(-\lambda x\right)$; }{\footnotesize \par}
{\footnotesize $0\le x\le+\infty$}  & {\footnotesize $\frac{1}{\lambda} - x$}  & -  \tabularnewline
{\footnotesize Gaussian }  &{\footnotesize$\{\mu,\sigma\}^{{\scriptsize T}}$} & {\footnotesize $\frac{1}{\sqrt{2\pi}\sigma}\exp\left[-\frac{1}{2}\left(\frac{x-\mu}{\sigma}\right)^{2}\right]$;}{\footnotesize \par}
{\footnotesize $-\infty\le x\le+\infty$} & {\footnotesize $\frac{1}{\sigma}\left(\frac{x-\mu}{\sigma}\right)$}  & {\footnotesize $\frac{1}{\sigma}\left[\left(\frac{x-\mu}{\sigma}\right)^{2}-1\right]$}  \tabularnewline
{\footnotesize Lognormal$^{\tnote{(a)}}$}  &{\footnotesize$\{\mu,\sigma\}^{{\scriptsize T}}$} & {\footnotesize $\frac{1}{\sqrt{2\pi}x\tilde{\sigma}}\exp\left[-\frac{1}{2}\left(\frac{\ln x-\tilde{\mu}}{\tilde{\sigma}}\right)^{2}\right]$; }{\footnotesize \par}

{\footnotesize $0<x\le+\infty$} & {\footnotesize $\begin{array}{c}
-\frac{1}{\tilde{\sigma}}\frac{\partial{\tilde{\sigma}}}{\partial{\mu}}+\frac{1}{\tilde{\sigma}^{2}}\left(\frac{\ln x-\tilde{\mu}}{\tilde{\sigma}}\right)\times\\
\left[\tilde{\sigma}\frac{\partial{\tilde{\mu}}}{\partial{\mu}}+\left(\ln x-\tilde{\mu}\right)\frac{\partial{\tilde{\sigma}}}{\partial{\mu}}\right]
\end{array}$}  & {\footnotesize $\begin{array}{c}
-\frac{1}{\tilde{\sigma}}\frac{\partial{\tilde{\sigma}}}{\partial{\sigma}}+\frac{1}{\tilde{\sigma}^{2}}\left(\frac{\ln x-\tilde{\mu}}{\tilde{\sigma}}\right)\times\\
\left[\tilde{\sigma}\frac{\partial{\tilde{\mu}}}{\partial{\sigma}}+\left(\ln x-\tilde{\mu}\right)\frac{\partial{\tilde{\sigma}}}{\partial{\sigma}}\right]
\end{array}$} \tabularnewline
{\footnotesize Truncated$^{\tnote{(b)}}$ Gaussian}  &{\footnotesize$\{\mu,\sigma\}^{{\scriptsize T}}$} & {\footnotesize $\begin{array}{c}
\frac{1}{\Phi(D)-\Phi(-D)}\frac{1}{\sqrt{2\pi}\sigma}\times\\
\exp\left[-\frac{1}{2}\left(\frac{x-\mu}{\sigma}\right)^{2}\right];
\end{array}$}{\footnotesize \par}

{\footnotesize $\mu-D\le x\le\mu+D$} & {\footnotesize $\frac{1}{\Phi(D)-\Phi(-D)}\frac{1}{\sigma}\left(\frac{x-\mu}{\sigma}\right)$}  & {\footnotesize $\frac{1}{\Phi(D)-\Phi(-D)}\frac{1}{\sigma}\left[\left(\frac{x-\mu}{\sigma}\right)^{2}-1\right]$} \tabularnewline
{\footnotesize Weibull}  & {\footnotesize$\{\lambda,k\}^{{\scriptsize T}}$} & {\footnotesize $\frac{k}{\lambda}\left(\frac{x}{\lambda}\right)^{k-1}\exp\left[-\left(\frac{x}{\lambda}\right)^{k}\right]$; }{\footnotesize \par}

{\footnotesize $0\le x\le+\infty$ } & {\footnotesize $\frac{k}{\lambda}\left[\left(\frac{x}{\lambda}\right)^{k}-1\right]$}  & {\footnotesize $\begin{array}{c}\frac{1}{k}+(\ln x-\ln\lambda)\times \\  \left[1-\left(\frac{x}{\lambda}\right)^{k}\right]\end{array}$} \tabularnewline
\hline
\end{tabular}
\renewcommand*\arraystretch{1.0}
\begin{tablenotes}
\footnotesize
\item [(a)] $\tilde{\sigma}^{2}=\ln\left(1+\sigma^{2}/\mu^{2}\right)$
and $\tilde{\mu}=\ln\mu-\tilde{\sigma}^{2}/2$. The partial derivatives
of $\tilde{\mu}$ and $\tilde{\sigma}$ with respect to $\mu$ or
$\sigma$ can be easily obtained, so they are not reported here.
\item [(b)] $\Phi(\cdot)$ is the cumulative distribution
function of a standard Gaussian variable; $D>0$ is a constant.
\end{tablenotes}
\end{centering}
\end{threeparttable}
\end{table}

\subsection{Example 1: A Trigonometric-Polynomial Function}

Consider the function
\begin{equation}
y(\mathbf{X})=\mathbf{a}_{1}^{T}\mathbf{X}+\mathbf{a}_{2}^{T}\sin\mathbf{X}+\mathbf{a}_{3}^{T}\cos\mathbf{X}+\mathbf{X}^{T}\mathbf{MX},\label{trigpoly}
\end{equation}
introduced by Oakley and O'Hagan \cite{oakley04}, where $\mathbf{X}=\{X_{1},\cdots,X_{15}\}^{T}\in\mathbb{R}^{15}$
is a $15$-dimensional Gaussian input vector with mean vector $\mathbb{E}[\mathbf{X}]=\{\mu,\cdots,\mu\}^{T}\in\mathbb{R}^{15}$
and covariance matrix $\mathbb{E}[\mathbf{XX}^{T}]=\sigma^{2}\mbox{diag}[1,\cdots,1]=:\sigma^{2}\mathbf{I}\in\mathbb{R}^{15\times15}$;
$\mathbf{d}=\{\mu,\sigma\}^{T}$; $\sin\mathbf{X}:=\{\sin X_{1},\cdots,\sin X_{15}\}^{T}\in\mathbb{R}^{15}$
and $\cos\mathbf{X}:=\{\cos X_{1},\cdots,\cos X_{15}\}^{T}\in\mathbb{R}^{15}$
are compact notations for $15$-dimensional vectors of sine and cosine
functions, respectively; and $\mathbf{a}_{i}\in\mathbb{R}^{15}$,
$i=1,2,3$, and $\mathbf{M}\in\mathbb{R}^{15\times15}$ are coefficient
vectors and matrix, respectively, obtained from Oakley and O'Hagan's
paper \cite{oakley04}. The objective of this example is to evaluate
the accuracy of the proposed PDD approximation in calculating the
sensitivities of the first two moments, $m^{(1)}(\mathbf{d}):=\mathbb{E}_{\mathbf{d}}[y(\mathbf{X})]$
and $m^{(2)}(\mathbf{d}):=\mathbb{E}_{\mathbf{d}}[y^{2}(\mathbf{X})]$,
with respect to the mean $\mu$ and standard deviation $\sigma$ of
$X_{i}$ at $\mathbf{d}_{0}=\{0,1\}^{T}$.

Figures \ref{fig:1ex1}(a) through \ref{fig:1ex1}(d) present the
plots of the relative errors in the approximate sensitivities, $\partial\tilde{m}_{S,m}^{(1)}(\mathbf{d}_{0})/\partial\mu$,
$\partial\tilde{m}_{S,m}^{(1)}(\mathbf{d}_{0})/\partial\sigma$, $\partial\tilde{m}_{S,m}^{(2)}(\mathbf{d}_{0})/\partial\mu$,
and $\partial\tilde{m}_{S,m}^{(2)}(\mathbf{d}_{0})/\partial\sigma$,
obtained by the proposed univariate and bivariate PDD methods (Equations
\ref{mom1spdd3} and \ref{mom2spdd3}) for increasing orders of orthonormal
polynomials, that is, when the PDD truncation parameters $S=1$ and
2, $1\le m\le8$, and $m'=2$. The measure-consistent Hermite polynomials
and associated Gauss-Hermite quadrature rule were used. The relative
error is defined as the ratio of the absolute difference between the
exact and approximate sensitivities, divided by the exact sensitivity,
where the exact sensitivity can be easily calculated for the function
$y$ in Equation \ref{trigpoly}. Although $y$ is a bivariate function
of $\mathbf{X}$, the sensitivities of the first moment by the univariate
and bivariate PDD approximations are identical for any $m$. This
is because the expectations of $\tilde{y}_{1,m}(\mathbf{X})$ and
$\tilde{y}_{2,m}(\mathbf{X})$, when $\mathbf{X}$ comprises independent
variables, are the same function of $\mathbf{d}$. In this case, the
errors committed by both PDD approximations drop at the same rate,
as depicted in Figures \ref{fig:1ex1}(a) and \ref{fig:1ex1}(b),
resulting in rapid convergence of the sensitivities of the first moment.
However, the same condition does not hold true
for the sensitivities of the second moment, because the univariate
and bivariate PDD approximations yield distinct sets of results. Furthermore,
the errors in the sensitivities of the second moment by the univariate
PDD approximation do not decay strictly monotonically, leveling off
when $m$ crosses a threshold, as displayed in Figures \ref{fig:1ex1}(c)
and \ref{fig:1ex1}(d). In contrast, the errors in the sensitivities
of the second moment by the bivariate PDD approximation attenuate
continuously with respect to $m$, demonstrating rapid convergence
of the proposed solutions. The numerical results presented are consistent
with the mean-square convergence of the sensitivities described by
Proposition \ref{p3}.
\begin{figure}[h]
\begin{centering}
\includegraphics[clip, scale=0.6]{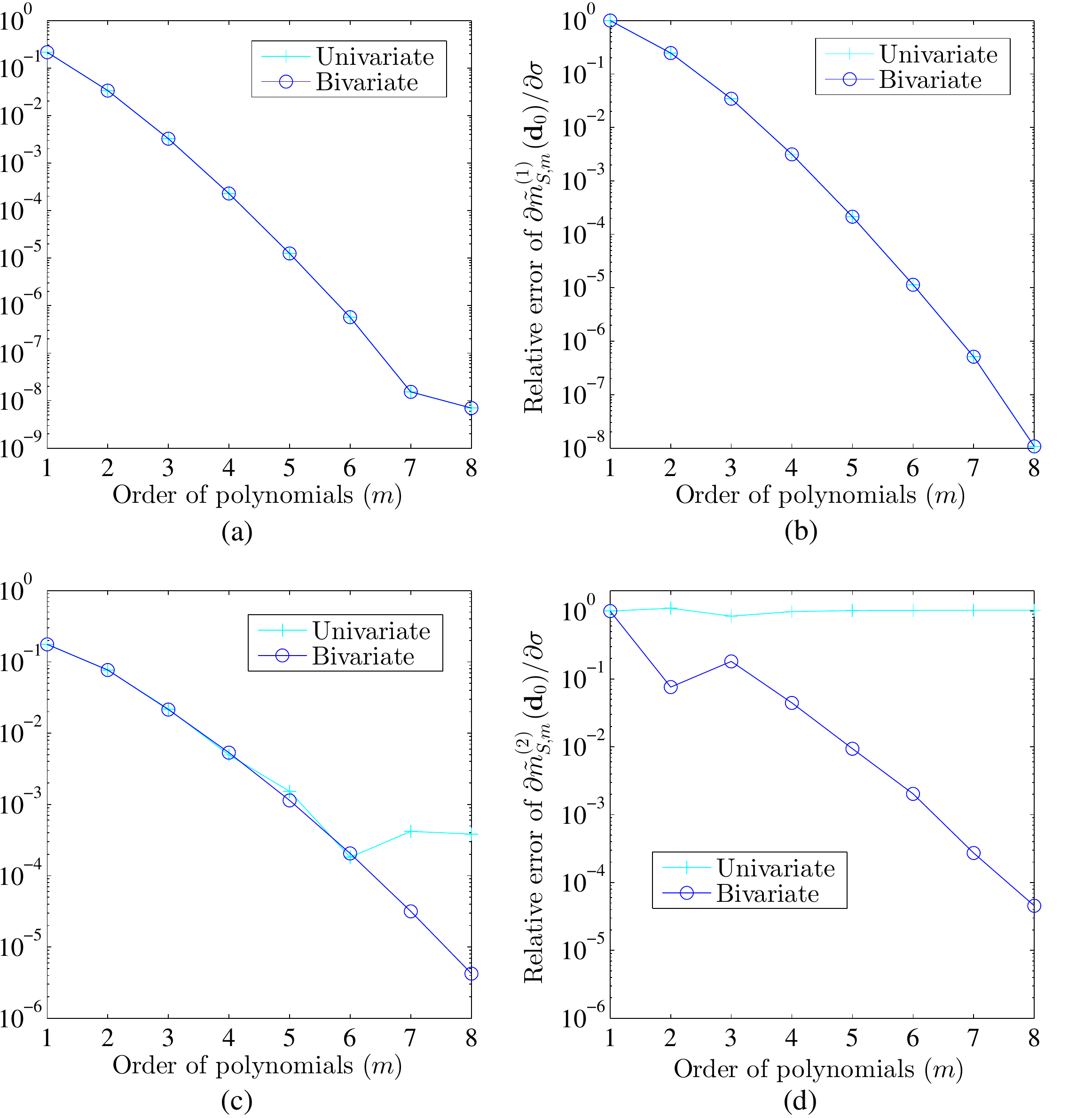}
\par\end{centering}
\caption{Relative errors in calculating the sensitivities of the first two
moments of $y$ due to various PDD truncations; (a) $\partial\tilde{m}_{S,m}^{(1)}(\mathbf{d}_{0})/\partial\mu$;
(b)$\partial\tilde{m}_{S,m}^{(1)}(\mathbf{d}_{0})/\partial\sigma$;
(c) $\partial\tilde{m}_{S,m}^{(2)}(\mathbf{d}_{0})/\partial\mu$;
(d) $\partial\tilde{m}_{S,m}^{(2)}(\mathbf{d}_{0})/\partial\sigma$ (Example 1)}
\label{fig:1ex1}
\end{figure}

\subsection{Example 2 : A Cubic Polynomial Function}

The second example is concerned with calculating the sensitivities
of the probability distribution of
\begin{equation}
y(\mathbf{X})=500-(X_{1}+X_{2})^{3}+X_{1}-X_{2}-X_{3}+X_{1}X_{2}X_{3}-X_{4},\label{cubic}
\end{equation}
where $X_{i}$, $i=1,2,3,4$, are four independent and identically
distributed random variables. The sensitivities were calculated by
the proposed PDD-MCS method using two approaches: (1) a direct approach
employing measure-consistent orthonormal polynomials as bases and
corresponding Gauss type quadrature rules for calculating the PDD
expansion coefficients, and (2) an indirect approach transforming
original random variables into Gaussian random variables, followed
by Hermite orthonormal polynomials as bases and the Gauss-Hermite
quadrature rule for calculating the expansion coefficients. Since
Equation \ref{cubic} represents a third-order polynomial, the measure-consistent
orthonormal polynomials with the largest order $m=3$ should exactly
reproduce $y$. In which case, the highest order of integrands for
calculating the PDD expansion coefficients is six; therefore, a four-point
($n=4$) measure-consistent Gauss quadrature should provide exact
values of the coefficients. In the direct approach, univariate ($S=1$),
bivariate ($S=2$), and trivariate ($S=3$) PDD approximations were
applied, where the expansion coefficients were calculated using $R=S$,
$m=3$, and $n=4$ in Equations \ref{pddcoeff1b} and \ref{pddcoeff2b}.
Therefore, the only source of error in a truncated PDD is the selection
of $S$. In the indirect approach, the transformation of $y$, if
the input variables follow non-Gaussian probability distributions,
leads to non-polynomials in the space of Gaussian variables; therefore,
approximation in a truncated PDD occur not only due to $S$, but also
due to $m$. Hence several values of $3\le m\le6$ were employed
for mappings into Gaussian variables. The coefficients in the indirect
approach were calculated by the $n$-point Gauss-Hermite quadrature
rule, where $n=m+1$.

A principal objective of this example is to gain insights on the choice
of orthonormal polynomials for solving this problem by PDD approximations.
Two distinct cases, depending on the probability distribution of input
variables, were studied.

\subsubsection{Case 1: Exponential Distributions}

For exponential distributions of input random variables, the PDF
\begin{equation}
f_{X_{i}}(x_{i};\lambda)=\left\{ \begin{array}{l@{\quad:\quad}l}
\lambda\exp(-\lambda x_{i})\; & x_{i}\ge0,\\
0 & x_{i}<0,
\end{array}\right.
\end{equation}
where $\lambda>0$ is the sole distribution parameter, $\mathbf{d}=\{\lambda\}\in\mathbb{R}$,
and $\mathbf{d}_{0}=\{1\}$.

Figure \ref{fig:2ex2-1}(a) presents the sensitivities of the probability
distribution of $y(\mathbf{X})$ with respect to $\lambda$ calculated
at $\mathbf{d}_{0}$ for different values of $\xi$ by the direct
approach. It contains four plots: one obtained from crude MCS/SF ($10^{6}$
samples) and the remaining three generated from univariate ($S=1$),
bivariate ($S=2$), and trivariate ($S=3$) PDD-MCS methods. For the
PDD-MCS methods, the measure-consistent Laguerre polynomials and associated
Gauss-Laguerre quadrature rule were used. The sensitivity of distributions,
all obtained for $m=3$, converge rapidly with respect to $S$. Compared
with crude MCS/SF, the univariate PDD-MCS method is less accurate
than others. This is due to the absence of cooperative effects of
random variables in the univariate approximation. The bivariate PDD-MCS
solution, which captures cooperative effects of any two variables,
is remarkably close to the crude Monte Carlo results. The results
from the trivariate decomposition and crude MCS/SF are coincident,
as $\tilde{y}_{3,3}(\mathbf{X})$ is identical to $y(\mathbf{X})$,
which itself is a trivariate function.

Using the indirect approach, Figures \ref{fig:2ex2-1}(b), \ref{fig:2ex2-1}(c),
and \ref{fig:2ex2-1}(d) depict the sensitivities of the distribution
of $y(\mathbf{X})$ by the univariate, bivariate, and trivariate PDD-MCS
methods for several values of $m$, calculated when the original variables
are transformed into standard Gaussian variables. The sensitivities
obtained by all three decomposition methods from the indirect approach
converge to the respective solutions from the direct approach when
$m$ and $n$ increase. However, the lowest order of Hermite polynomials
required to converge in the indirect approach is six, a number twice
that employed in the direct approach employing Laguerre polynomials.
This is due to higher nonlinearity of the mapped $y$ induced by the
transformation from exponential to Gaussian variables. Clearly, the
direct approach employing Laguerre polynomials and the Gauss-Laguerre
quadrature rule is the preferred choice for calculating sensitivities
of the probability distribution by the PDD-MCS method.
\begin{figure}[tbph]
\begin{centering}
\includegraphics[clip, scale=0.58]{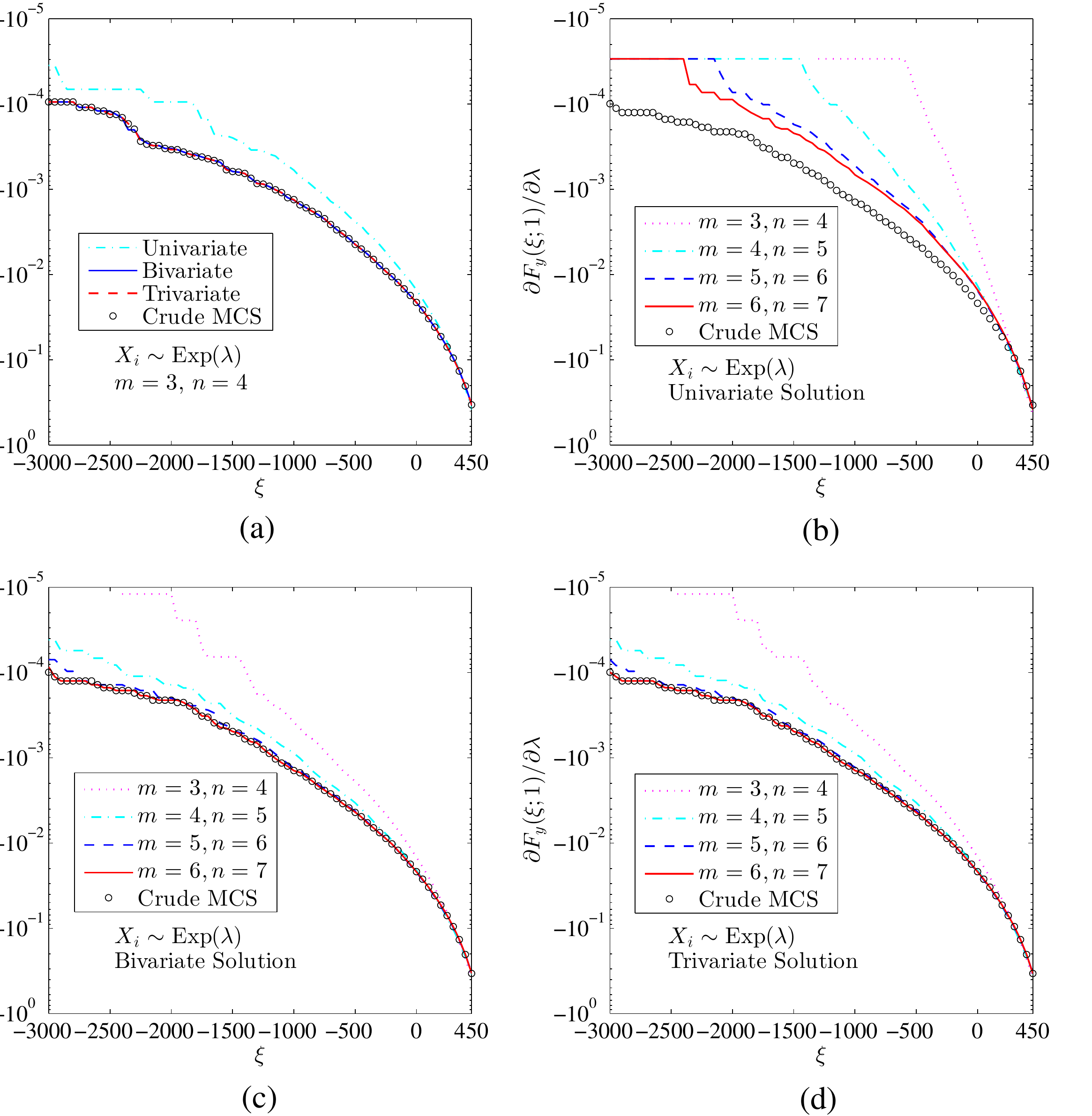}
\end{centering}
\caption{Sensitivities of the probability distribution of $y$ with respect
to $\lambda$ for exponential distributions of input
variables; (a) direct approach; (b) indirect approach-univariate;
(c) indirect approach-bivariate; (d) indirect approach-trivariate (Example 2)}
\label{fig:2ex2-1}
\end{figure}

\subsubsection{Case 2: Weibull Distributions}

For Weibull distributions of input random variables, the PDF
\begin{equation}
f_{X_{i}}(x_{i};\lambda,k)=\left\{ \begin{array}{l@{\quad:\quad}l}
{\displaystyle \frac{k}{\lambda}}\left({\displaystyle \frac{x_{i}}{\lambda}}\right)^{k-1}\exp\left[-\left({\displaystyle \frac{x_{i}}{\lambda}}\right)^{k}\right]\; & x_{i}\ge0,\\
0 & x_{i}<0,
\end{array}\right.
\end{equation}
where $\lambda>0$ and $k>0$ are scale and shape distribution parameters,
respectively, $\mathbf{d}=\{\lambda,k\}^{T}\in\mathbb{R}^{2}$, and
$\mathbf{d}_{0}=\{1,0.5\}^{T}$.

The sensitivities of the probability distribution of $y(\mathbf{X})$
with respect to $\lambda$ and $k$, calculated by the direct approach,
at $\mathbf{d}_{0}$ is exhibited in Figures \ref{fig:4ex2-2}(a)
and \ref{fig:3ex2-3}(a), respectively. Again, four plots, comprising
the solutions from crude MCS/SF ($10^{6}$ samples) and three PDD-MCS
methods using the direct approach, are illustrated. Since classical
orthonormal polynomials do not exist for Weibull probability measures,
the Stieltjes procedure was employed to numerically
determine the measure-consistent orthonormal polynomials and corresponding
Gauss quadrature formula \cite{rahman09b}. Similar to Case 1, both
sensitivities of the distribution by the PDD-MCS method in Figures \ref{fig:4ex2-2}(a)
and \ref{fig:3ex2-3}(a), all obtained for $m=3$, converge rapidly
to crude MCS solutions with respect to $S$. However, the sensitivities
of the distribution by all three PDD-MCS approximations, when calculated
using the indirect approach and shown in Figures \ref{fig:4ex2-2}(b)
through \ref{fig:4ex2-2}(d) and Figures \ref{fig:3ex2-3}(b) through
\ref{fig:3ex2-3}(d), fail to get closer even when the order of Hermite
polynomials is twice that employed in the direct approach. The lack
of convergence is attributed to a significantly higher nonlinearity
of the transformation from Weibull to Gaussian variables than that
from exponential to Gaussian variables. Therefore, a direct approach
entailing measure-consistent orthogonal polynomials and associated
Gauss quadrature rule, even in the absence of classical polynomials,
is desirable for generating both accurate and efficient solutions
by the PDD-MCS method.
\begin{figure}[tbph]
\begin{centering}
\includegraphics[clip, scale=0.58]{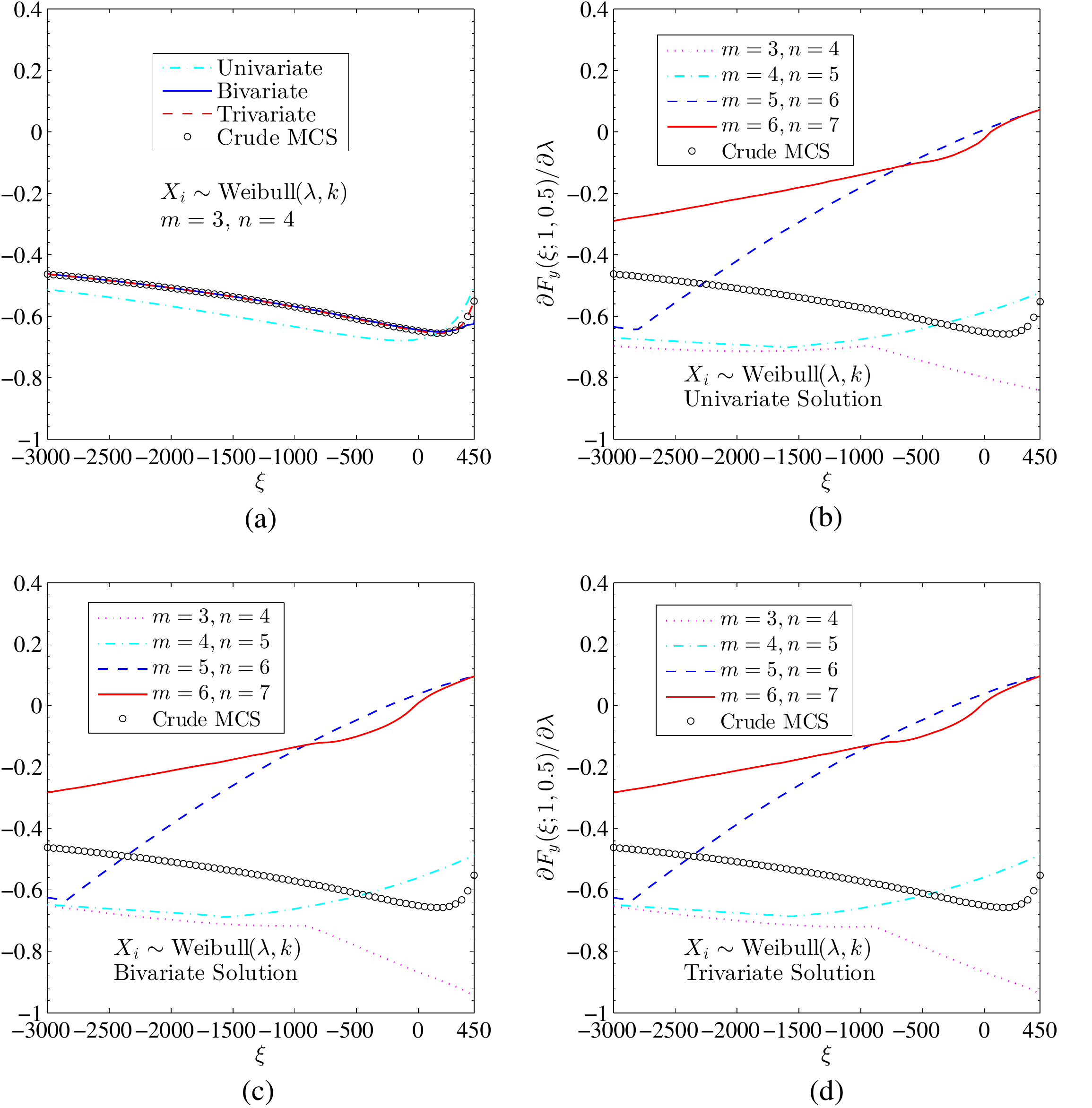}
\par\end{centering}
\caption{Sensitivities of the probability distribution of $y$ with respect
to $\lambda$ for Weibull distributions of input variables;
(a) direct approach; (b) indirect approach-univariate; (c) indirect
approach-bivariate; (d) indirect approach-trivariate (Example 2)}
\label{fig:4ex2-2}
\end{figure}
\begin{figure}[tbph]
\begin{centering}
\includegraphics[clip, scale=0.58]{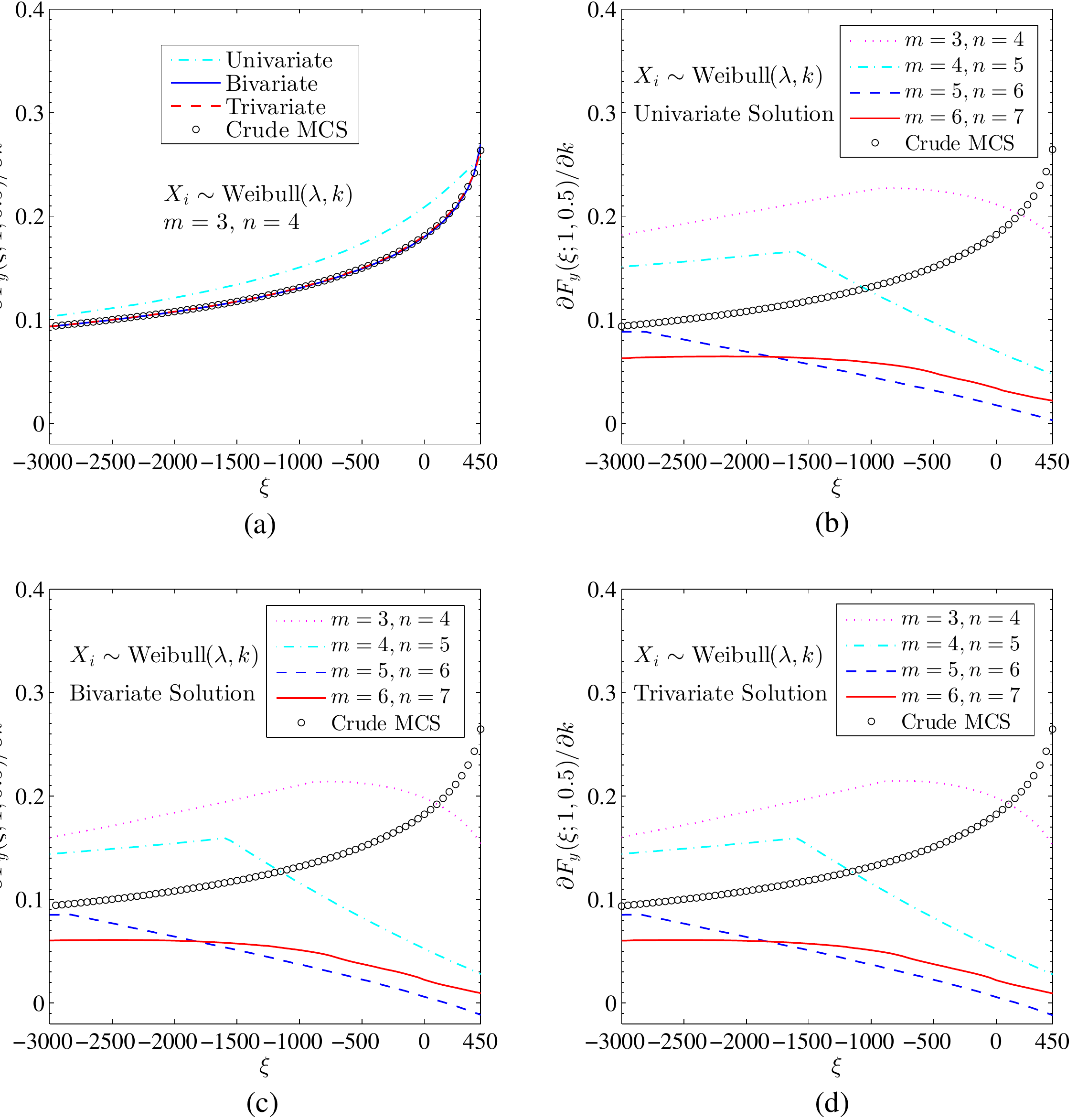}
\par\end{centering}

\caption{Sensitivities of the probability distribution of $y$ with respect
to $k$ for Weibull distributions of input variables;
(a) direct approach; (b) indirect approach-univariate; (c) indirect
approach-bivariate; (d) indirect approach-trivariate (Example 2)}
\label{fig:3ex2-3}
\end{figure}

\subsection{Example 3: A Function of Gaussian Variables}

Consider a component reliability problem with the performance function
\begin{equation}
y(\mathbf{X})=\frac{1}{1000+\sum\limits _{i=1}^{N}X_{i}}-\frac{1}{1000+3\sqrt{N}},\label{fgauss}
\end{equation}
where $\mathbf{X}\sim N(\boldsymbol{\mu},\boldsymbol{\Sigma})$ is
an $N$-dimensional Gaussian random vector with mean vector $\boldsymbol{\mu}=\{\mu,\cdots,\mu\}^{T}$
and covariance matrix $\boldsymbol{\Sigma}=\sigma^{2}\mbox{diag}[1,\cdots,1]=:\sigma^{2}\mathbf{I}$,
and $\mathbf{d}=\{\mu,\sigma\}^{T}$. The objective of this example
is to evaluate the accuracy of the proposed PDD-SPA and PDD-MCS methods
in calculating the failure probability $P_{F}(\mathbf{d}):=P_{\mathbf{d}}[y(\mathbf{X})<0]$
and its sensitivities $\partial P_{F}(\mathbf{d}_{0})\left/\partial\mu\right.$
and $\partial P_{F}(\mathbf{d}_{0})\left/\partial\sigma\right.$ at
$\mathbf{d}_{0}=\{0,1\}^{T}$ for two problem sizes or dimensions:
$N=10$ and $N=100$. The exact solutions for a general $N$-dimensional
problem are
\begin{equation}
P_{F}(\mathbf{d})=\Phi(-\beta),\;\frac{\partial P_{F}(\mathbf{d})}{\partial\mu}=\frac{\phi(-\beta)\sqrt{N}}{\sigma},\;\frac{\partial P_{F}(\mathbf{d})}{\partial\sigma}=\frac{\phi(-\beta)(3-\mu\sqrt{N})}{\sigma^{2}},
\end{equation}
where $\beta=(3-\mu\sqrt{N})\left/\sigma\right.$, provided that $0<\sigma^{2}<\infty$.

Since $y$ in Equation \ref{fgauss} is a non-polynomial function,
the univariate ($S=1$) or bivariate ($S=2$) truncation of PDD for
a finite value of $m$, regardless how large, provides only an approximation.
Nonetheless, using only $m=3$ and $n=4$, the univariate and bivariate
estimates of the failure probability and its two sensitivities by
the PDD-SPA and PDD-MCS methods for $N=10$ are listed in Table \ref{tab:3-N10}.
The measure-consistent Hermite polynomials and associated Gauss-Hermite
quadrature rule were used in both methods. The results of the PDD-SPA
method are further broken down according to Options I (Equation \ref{momrsen3})
and II (Equation \ref{momrsen4}) for calculating all moments of order
up to four to approximate the CGF of $y(\mathbf{X})$, as explained in Algorithm \ref{alg:PDD-SPA}. Option I requires
at most eight-dimensional integrations in the bivariate PDD-SPA method
for calculating the moments of $y(\mathbf{X})$, whereas Option II
entails at most two-dimensional integrations for the values of $\bar{S}=2$
and $\bar{m}=6$ selected. However, the differences between the two
respective estimates of the failure probability and its sensitivities
by these options, in conjunction with either the univariate or the bivariate
PDD approximation, are negligibly small. Therefore, Option II is not
only accurate, but also facilitates efficient solutions by the PDD-SPA
method, at least in this example. Compared with the results of crude
MCS/SF ($10^{6}$ samples) or the exact solution, also listed in Table
\ref{tab:3-N10}, both univariate and bivariate versions of the PDD-SPA
method, regardless of the option, are satisfactory. The same trend
holds for the univariate and bivariate PDD-MCS methods. No meaningful
difference is found between the respective accuracies of the PDD-SPA
and PDD-MCS solutions for a given truncation $S$. Indeed, the agreement
between the bivariate solutions from the PDD-SPA or PDD-MCS method
and the benchmark results is excellent.

\begin{table}[tbph]
\caption{Component failure probability and sensitivities at $\mathbf{d}_{0}=\{0,1\}^{T}$ for $N=10$ (Example 3)}
\label{tab:3-N10}
\begin{centering}
\begin{tabular}{>{\centering}p{0.64in}>{\centering}p{0.48in}>{\centering}p{0.48in}>{\centering}p{0.48in}>{\centering}p{0.48in}>{\centering}p{0.53in}>{\centering}p{0.53in}>{\centering}p{0.32in}>{\centering}p{0.25in}}
\hline
 &{\footnotesize PDD-SPA}{\footnotesize \par}{\footnotesize (Univariate, Option I)}&{\footnotesize PDD-SPA}{\footnotesize \par}{\footnotesize (Univariate, Option II)}&{\footnotesize PDD-SPA}{\footnotesize \par}{\footnotesize (Bivariate, Option I)}&{\footnotesize PDD-SPA}{\footnotesize \par}{\footnotesize (Bivariate, Option II)}&{\footnotesize PDD-MCS}{\footnotesize \par}{\footnotesize (Univariate)}&{\footnotesize PDD-MCS}{\footnotesize \par}{\footnotesize (Bivariate)}&{\footnotesize Crude MCS/SF}&{\footnotesize Exact}\tabularnewline
\hline
{\scriptsize$P_{F}(\mathbf{d}_{0})$\par}{\scriptsize$(\times10^{-3})$}&{\scriptsize$1.349$}&{\scriptsize$1.453$}&{\scriptsize$1.349$}&{\scriptsize$1.347$}&{\scriptsize$1.510$}&{\scriptsize$1.397$}&{\scriptsize$1.397$}&{\scriptsize$1.350$}\tabularnewline
{\scriptsize$\partial{P_{F}(\mathbf{d}_{0})}/\partial{\mu}$\par}{\scriptsize$(\times10^{-2})$}&{\scriptsize$1.401$}&{\scriptsize$1.529$}&{\scriptsize$1.401$}&{\scriptsize$1.550$}&{\scriptsize$1.553$}&{\scriptsize$1.447$}&{\scriptsize$1.447$}&{\scriptsize$1.401$}\tabularnewline
{\scriptsize$\partial{P_{F}(\mathbf{d}_{0})}/\partial{\sigma}$\par}{\scriptsize$(\times10^{-2})$}&{\scriptsize$1.330$}&{\scriptsize$1.409$}&{\scriptsize$1.330$}&{\scriptsize$1.326$}&{\scriptsize$1.472$}&{\scriptsize$1.371$}&{\scriptsize$1.371$}&{\scriptsize$1.330$}\tabularnewline
{\scriptsize No. of function eval.}&{\scriptsize 41}&{\scriptsize 41}&{\scriptsize 761}&{\scriptsize 761}&{\scriptsize 41}&{\scriptsize 761}&{\scriptsize$10^{6}$}&{\scriptsize -}\tabularnewline
\hline
\end{tabular}
\par\end{centering}
\end{table}

For high-dimensional problems, such as $N=100$, Table \ref{tab:4-N100}
summarizes the estimates of the failure probability and its sensitivities
by the PDD-SPA and PDD-MCS methods using $m=3$. Due to the higher
dimension, the PDD-SPA method with Option I requires numerous eight-dimensional
integrations for calculating moments of $y(\mathbf{X})$ and is no
longer practical. Therefore, the PDD-SPA method with Option II requiring
only two-dimensional $(\bar{S}=2,\ \bar{m}=6)$ integrations was used
for $N=100$. Again both univariate and bivariate approximations were
invoked for the PDD-SPA and PDD-MCS methods. Compared with the benchmark
results of crude MCS/SF ($10^{6}$ samples) or the exact solution,
listed in Table \ref{tab:4-N100}, the bivariate PDD-SPA method or
the bivariate PDD-MCS method provides highly accurate solutions for
this high-dimensional reliability problem.

\begin{table}[tbph]
\caption{Component failure probability and sensitivities at $\mathbf{d}_{0}=\{0,1\}^{T}$ for $N=100$ (Example 3)}
\label{tab:4-N100}
\begin{centering}
\begin{tabular}{>{\centering}p{0.64in}>{\centering}p{0.5in}>{\centering}p{0.5in}>{\centering}p{0.53in}>{\centering}p{0.53in}>{\centering}p{0.5in}>{\centering}p{0.5in}}
\hline
 & {\footnotesize PDD-SPA }{\footnotesize \par}

{\footnotesize (Univariate, Option II)}  & {\footnotesize PDD-SPA }{\footnotesize \par}

{\footnotesize (Bivariate, Option II)}  & {\footnotesize PDD-MCS }{\footnotesize \par}

{\footnotesize (Univariate)}  & {\footnotesize PDD-MCS }{\footnotesize \par}

{\footnotesize (Bivariate)}  & {\footnotesize Crude MCS/SF}  & {\footnotesize Exact} \tabularnewline
\hline
{\scriptsize$P_{F}(\mathbf{d}_{0})$\par}{\scriptsize$(\times10^{-3})$}&{\scriptsize$1.731$}&{\scriptsize$1.320$}&{\scriptsize $1.724$}& {\scriptsize$1.344$}&{\scriptsize$1.352$}&{\scriptsize$1.350$}\tabularnewline
{\scriptsize$\partial{P_{F}(\mathbf{d}_{0})}/\partial{\mu}$\par}{\scriptsize$(\times10^{-2})$}&{\scriptsize$5.994$}&{\scriptsize$6.412$}&{\scriptsize$5.538$}&{\scriptsize$4.413$}&{\scriptsize$4.437$}&{\scriptsize$4.432$}\tabularnewline
{\scriptsize$\partial{P_{F}(\mathbf{d}_{0})}/\partial{\sigma}$\par}{\scriptsize$(\times10^{-2})$}&{\scriptsize$1.612$}&{\scriptsize $1.277$}&{\scriptsize$1.556$}&{\scriptsize$1.291$}&{\scriptsize $1.302$}&{\scriptsize$1.330$}\tabularnewline
{\scriptsize No. of function eval.}&{\scriptsize 401}&{\scriptsize 79,601} &{\scriptsize 401} &{\scriptsize 79,601} &{\scriptsize$10^{6}$} &{\scriptsize -}\tabularnewline
\hline
\end{tabular}
\par\end{centering}
\end{table}

Tables \ref{tab:3-N10} and \ref{tab:4-N100} also specify the relative
computational efforts of the PDD-SPA and PDD-MCS methods, measured
in terms of numbers of original function evaluations, when $N=10$
and $N=100$. Given the truncation parameter $S$, the PDD-SPA and
PDD-MCS methods require identical numbers of function evaluations,
meaning that their computational costs are practically the same. Although
the bivariate approximation is significantly more expensive than the
univariate approximation, the former generates highly accurate solutions,
as expected. However, both versions of the PDD-SPA or PDD-MCS method
are markedly more economical than the crude MCS/SF method for solving
this high-dimensional reliability problem.

\subsection{Example 4 : A Function of Non-Gaussian Variables}

Consider the univariate function \cite{huang12}
\begin{equation}
y(\mathbf{X})=X_{1}+2X_{2}+2X_{3}+X_{4}-5X_{5}-5X_{6}
\end{equation}
of six statistically independent and lognormally distributed random
variables $X_{i}$ with means $\mu_{i}$ and standard deviations $c\mu_{i}$,
$i=1,\cdots,6$, where $c>0$ is a constant, representing the coefficient
of variation of $X_{i}$. The design vector $\mathbf{d}=\{\mu_{1},\cdots\mu_{6},\sigma_{1},\cdots,\sigma_{6}\}^{T}$.
The objective of this example is to evaluate the accuracy of the proposed
PDD-SPA method in estimating the failure probability $P_{F}(\mathbf{d}):=P_{\mathbf{d}}[y(\mathbf{X})<0]$
and its sensitivities $\partial P_{F}(\mathbf{d})\left/\partial\mu_{i}\right.$
and $\partial P_{F}(\mathbf{d})\left/\partial\sigma_{i}\right.$,
$i=1,\cdots,6$, at $\mathbf{d}=\mathbf{d}_{0}=\{120,120,120,120,50,40,120c,120c,120c,120c,50c,40c\}^{T}$
for $0.1\le c\le0.7$.

The function $y$, being both univariate and linear, is exactly reproduced
by the univariate ($S=1$), first-order ($m=1$) PDD approximation
when orthonormal polynomials consistent with lognormal probability
measures are used. Therefore, the univariate, first-order PDD approximation,
along with Option I (Equation \ref{momrsen3}), was employed in the
PDD-SPA method to approximate $P_{F}(\mathbf{d}_{0})$, $\partial P_{F}(\mathbf{d}_{0})\left/\partial\mu_{i}\right.$, and $\partial P_{F}(\mathbf{d}_{0})\left/\partial\sigma_{i}\right.$. All moments of order up to four were estimated according to Algorithm \ref{alg:PDD-SPA}. The measure-consistent solutions by the PDD-SPA method and crude MCS/SF
are presented in Figures \ref{fig:5ex4}(a), \ref{fig:5ex4}(b), and
\ref{fig:5ex4}(c). Huang and Zhang \cite{huang12}, who solved the
same problem, reported similar results, but at the expense of higher-order
integrations stemming from transformation to Gaussian variables. No
such transformation was required or performed in this work. According
to Figure \ref{fig:5ex4}(a), the failure probability curve generated
by the PDD-SPA method closely traces the path of crude MCS/SF ($10^{6}$
samples) for low coefficients of variation, although a slight deviation
begins to appear when $c$ exceeds about 0.4. The loss of accuracy
becomes more pronounced when comparing the sensitivities of the failure
probability with respect to means and standard deviations in Figures
\ref{fig:5ex4}(b) and \ref{fig:5ex4}(c). Indeed, for large coefficients
of variation, that is, for $c>0.4$, some of the sensitivities are
no longer accurately calculated by the PDD-SPA method. This is because
the fourth-order ($Q=4)$ approximation of the CGF of $y(\mathbf{X})$,
used for constructing the PDD-SPA method, is inadequate. Indeed, Table
\ref{tab:5-cgfError} reveals that the relative errors in the fourth-order
Taylor approximation of the CGF, obtained by MCS ($10^{8}$ samples)
and evaluated at respective saddlepoints, rises with increasing values
of the coefficient of variation from $0.2$ to $0.7$. Therefore,
a truncation larger than four is warranted for higher-order approximations
of CGF, but doing so engenders an added difficulty in finding a unique
saddlepoint. The topic merits further study.

\begin{table}[tbph]
\centering
\caption{Relative errors in calculating CGF (Example 4)}
\label{tab:5-cgfError}
\begin{threeparttable}
\begin{centering}
\renewcommand*\arraystretch{1.2}
\begin{tabular}{ccc}
\hline
$c$ & $t_{s}$  & Relative error $^{\tnote{(a)}}$ \tabularnewline
\hline
$0.1$  & $-1.0029\times10^{-1}$  & $0.0248$ \tabularnewline
$0.2$  & $-2.5008\times10^{-2}$  & $0.0068$ \tabularnewline
$0.3$  & $-1.1066\times10^{-2}$  & $0.0125$ \tabularnewline
$0.4$  & $-6.1850\times10^{-3}$  & $0.0183$ \tabularnewline
$0.5$  & $-3.9250\times10^{-3}$  & $0.0329$ \tabularnewline
$0.6$  & $-2.6966\times10^{-3}$  & $0.0447$ \tabularnewline
$0.7$  & $-1.9551\times10^{-3}$  & $0.2781$ \tabularnewline
\hline
\end{tabular}
\renewcommand*\arraystretch{1.0}
\begin{tablenotes}
\footnotesize
\item [(a)] The sample size of MCS is $10^{8}$.
\end{tablenotes}
\par\end{centering}
\end{threeparttable}
\end{table}

It is important to note that the univariate, first-order PDD-MCS method,
employing measure-consistent orthonormal polynomials, should render
the same solution of crude MCS/SF. This is the primary reason why
the PDD-MCS results are not depicted in Figures \ref{fig:5ex4}(a)
through \ref{fig:5ex4}(c). Nonetheless, the PDD-MCS method should
be more accurate than the PDD-SPA method in solving this problem,
especially at larger coefficients of variation.
\begin{figure}[tbph]
\begin{centering}
\includegraphics[clip,scale=0.62]{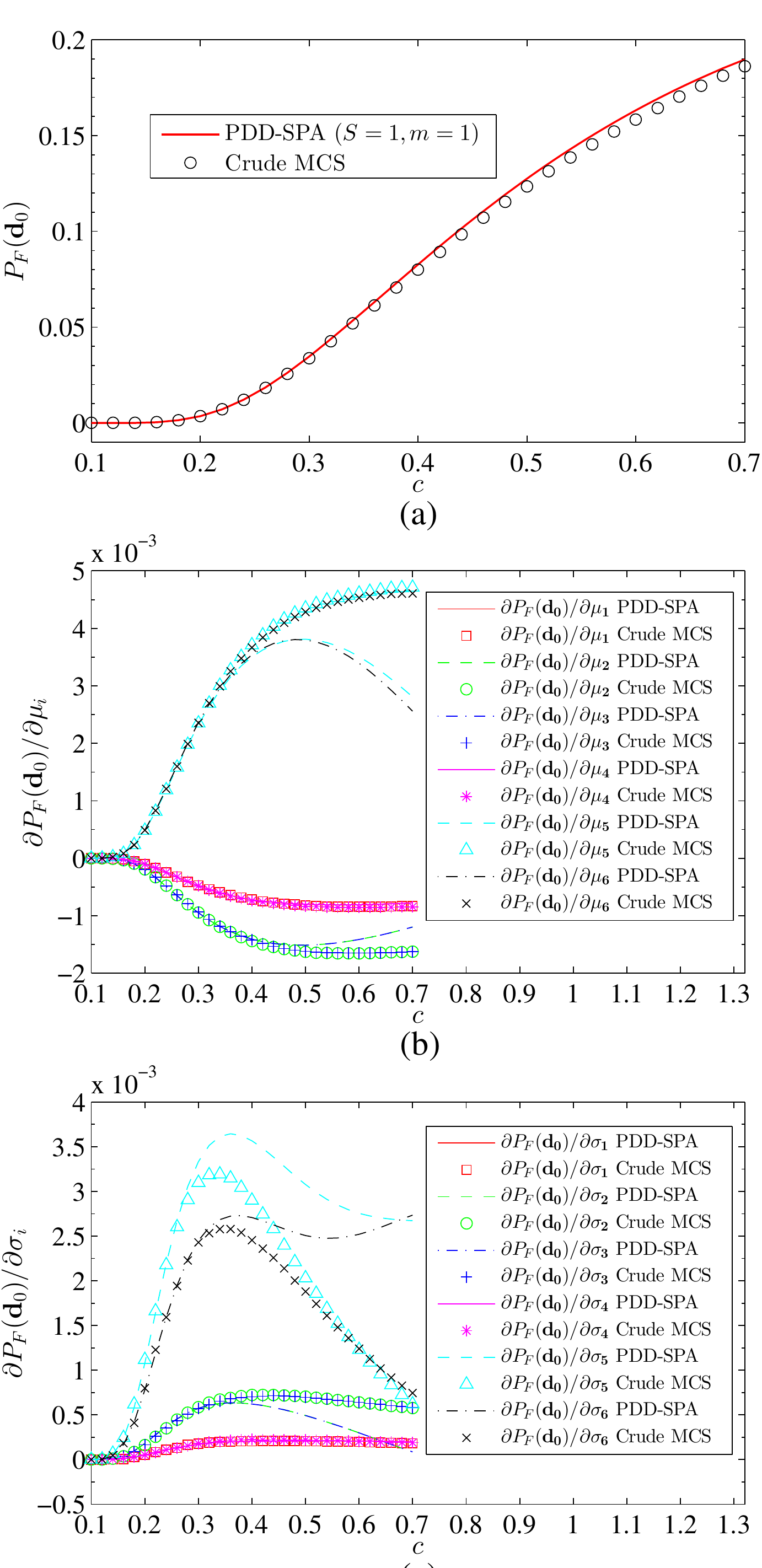}
\par\end{centering}
\caption{Results of the reliability and sensitivity analyses by
the PDD-SPA method and crude MCS/SF; (a) failure probability; (b)
sensitivities of failure probability with respect to means; (c) sensitivities
of failure probability with respect to standard deviations (Example 4)}
\label{fig:5ex4}
\end{figure}

\subsection{Example 5: A Six-Bay, Twenty-One-Bar Truss}

This example demonstrates how system reliability and its sensitivities
can be efficiently estimated with the PDD-MCS method. A linear-elastic,
six-bay, twenty-one-bar truss structure, with geometric properties
shown in Figure \ref{fig:6ex5-21bar}, is simply supported at nodes
1 and 12, and is subjected to four concentrated loads of 10,000 lb
(44,482 N) at nodes 3, 5, 9, and 11 and a concentrated load of 16,000
lb (71,172 N) at node 7. The truss material is made of an aluminum
alloy with the Young's modulus $E=10^{7}$ psi (68.94 GPa). The random
input is $\mathbf{X}=\{X_{1},\cdots,X_{21}\}^{T}\in\mathbb{R}^{21}$,
where $X_{i}$ is the cross-sectional areas of the $i$th truss member.
The random variables are independent and lognormally distributed with
means $\mu_{i}$, $i=1,\cdots,21$, each of which has a ten percent
coefficient of variation. From linear-elastic finite-element analysis
(FEA), the maximum vertical displacement $v_{\max}(\mathbf{X})$ and
maximum axial stress $\sigma_{\max}(\mathbf{X})$ occur at node $7$
and member 3 or 4, respectively, where the permissible displacement
and stress are limited to $d_{\text{allow}}=0.266$ in (6.76 mm) and
$\sigma_{\text{allow}}=37,680$ psi (259.8 MPa), respectively. The
system-level failure set is defined as $\Omega_{F}:=\{\mathbf{x}:\{y_{1}(\mathbf{x})<0\}\cup\{y_{2}(\mathbf{x})<0\}\}$,
where the performance functions
\begin{equation}
y_{1}(\mathbf{X})=1-\frac{|v_{\max}(\mathbf{X})|}{d_{\text{allow}}},\ y_{2}(\mathbf{X})=1-\frac{|\sigma_{\max}(\mathbf{X})|}{\sigma_{\text{allow}}}.
\end{equation}
The design vector is $\mathbf{d}=\{\mu_{1},\cdots,\mu_{21}\}^{T}$. The
objective of this example is to evaluate the accuracy of the proposed
PDD-MCS method in estimating the system failure probability $P_{F}(\mathbf{d}):=P_{\mathbf{d}}\left[\{y_{1}(\mathbf{X})<0\}\cup\{y_{2}(\mathbf{X})<0\}\right]$
and its sensitivities $\partial{P_{F}(\mathbf{d})}/\partial{\mu_{i}},\ i=1,\dots,21$
at $\mathbf{d}=\mathbf{d}_{0}=\{2,2,2,2,2,2,10,10,10,10,10,10,3,3,3,3,3,1,1,1,1\}^{T}\ \mbox{in}^{2}$
($\times2.54^{2}\;\mathrm{cm^{2}}$).
\begin{figure}[tbph]
\begin{centering}
\includegraphics[clip,scale=0.75]{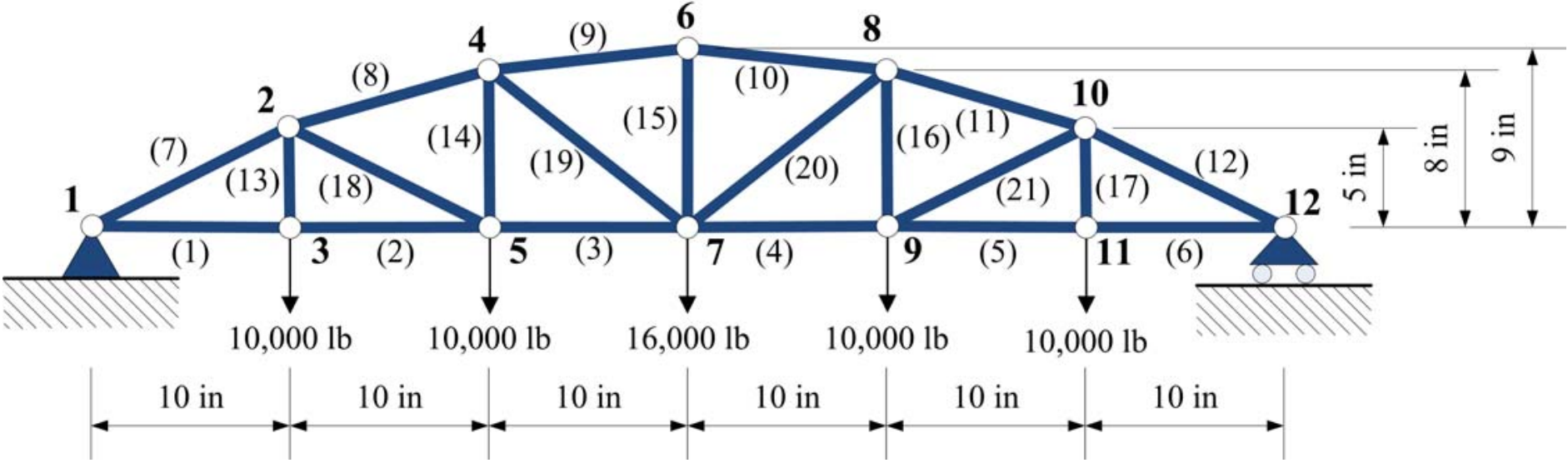}
\par\end{centering}
\caption{A six-bay, twenty-one-bar truss structure (Example 5)}
\label{fig:6ex5-21bar}
\end{figure}

Table \ref{tab:6-21bar} presents the system failure probability and
its 21 sensitivities obtained using the bivariate ($S=2)$, third-order
($m=3$) PDD approximations of $y_{1}(\mathbf{X})$ and $y_{2}(\mathbf{X})$
and two versions of crude MCS: crude MCS/SF and crude MCS/FD, providing
benchmark solutions. The crude MCS/FD method does not depend on score
functions and, therefore, facilitates an independent verification
of the PDD-MCS method. The respective sensitivities obtained by the
PDD-MCS method and crude MCS/SF are practically the same. However,
crude MCS/FD typically gives biased sensitivity estimates, where slight
fluctuations in the results are expected due to a finite variance
of the estimator. For two instances, such as when the sensitivities
are too small, crude MCS/FD produces trivial solutions and hence cannot
be used as reference solutions. Nonetheless, the general quality of
agreement between the results of the PDD-MCS method and crude MCS/FD
is very good. Comparing the computational efforts, only 3445 FEA were
required to produce the results of the PDD-MCS method in Table \ref{tab:6-21bar},
whereas $10^{7}$ and $22\times10^{7}$ FEA (samples) were incurred
by crude MCS/SF and crude MCS/FD, respectively. The 22-fold increase
in the number of FEA in crude MCS/FD is due to forward finite-difference
calculations entailing all 21 sensitivities. Therefore, the PDD-MCS
method provides not only highly accurate, but also vastly efficient,
solutions of system reliability problems.

\begin{table}[tbph]
\begin{centering}
\caption{System failure probability and sensitivities for the six-bay, twenty-one-bar
truss (Example 5)}
\label{tab:6-21bar}
\renewcommand*\arraystretch{1.1}
\begin{tabular}{>{\centering}p{3cm}>{\centering}p{3cm}>{\centering}p{3cm}>{\centering}p{3cm}}
\hline
 & PDD-MCS  & Crude MCS/SF  & Crude MCS/FD \tabularnewline
\hline
$P_{F}(\mathbf{d}_{0})$  & $\ \ \ 8.1782\times10^{-3}$  & $\ \ \ 8.3890\times10^{-3}$  & $\ \ \ 8.3890\times10^{-3}$ \tabularnewline
$\partial{P_{F}}(\mathbf{d}_{0})/\partial{\mu_{1}}$  & $-2.6390\times10^{-2}$  & $-2.6546\times10^{-2}$  & $-2.3895\times10^{-2}$ \tabularnewline
$\partial{P_{F}}(\mathbf{d}_{0})/\partial{\mu_{2}}$  & $-2.6385\times10^{-2}$  & $-2.6505\times10^{-2}$  & $-2.3810\times10^{-2}$ \tabularnewline
$\partial{P_{F}}(\mathbf{d}_{0})/\partial{\mu_{3}}$  & $-1.0010\times10^{-1}$  & $-1.0320\times10^{-1}$  & $-8.8875\times10^{-2}$ \tabularnewline
$\partial{P_{F}}(\mathbf{d}_{0})/\partial{\mu_{4}}$  & $-3.5684\times10^{-2}$  & $-3.5972\times10^{-2}$  & $-3.1960\times10^{-2}$ \tabularnewline
$\partial{P_{F}}(\mathbf{d}_{0})/\partial{\mu_{5}}$  & $-2.6356\times10^{-2}$  & $-2.6469\times10^{-2}$  & $-2.3825\times10^{-2}$ \tabularnewline
$\partial{P_{F}}(\mathbf{d}_{0})/\partial{\mu_{6}}$  & $-2.6266\times10^{-2}$  & $-2.6364\times10^{-2}$  & $-2.3950\times10^{-2}$ \tabularnewline
$\partial{P_{F}}(\mathbf{d}_{0})/\partial{\mu_{7}}$  & $-1.3189\times10^{-3}$  & $-1.3213\times10^{-3}$  & $-1.1970\times10^{-3}$ \tabularnewline
$\partial{P_{F}}(\mathbf{d}_{0})/\partial{\mu_{8}}$  & $-1.3294\times10^{-3}$  & $-1.3244\times10^{-3}$  & $-1.2820\times10^{-3}$ \tabularnewline
$\partial{P_{F}}(\mathbf{d}_{0})/\partial{\mu_{9}}$  & $-1.6665\times10^{-3}$  & $-1.6514\times10^{-3}$  & $-1.5610\times10^{-3}$ \tabularnewline
$\partial{P_{F}}(\mathbf{d}_{0})/\partial{\mu_{10}}$  & $-1.7554\times10^{-3}$  & $-1.7576\times10^{-3}$  & $-1.5670\times10^{-3}$ \tabularnewline
$\partial{P_{F}}(\mathbf{d}_{0})/\partial{\mu_{11}}$  & $-1.3892\times10^{-3}$  & $-1.3945\times10^{-3}$  & $-1.2530\times10^{-3}$ \tabularnewline
$\partial{P_{F}}(\mathbf{d}_{0})/\partial{\mu_{12}}$  & $-1.3136\times10^{-3}$  & $-1.3140\times10^{-3}$  & $-1.2060\times10^{-3}$ \tabularnewline
$\partial{P_{F}}(\mathbf{d}_{0})/\partial{\mu_{13}}$  & $\ \ \ 9.1378\times10^{-5}$  & $\ \ \ 7.2857\times10^{-5}$  & $0.0$ \tabularnewline
$\partial{P_{F}}(\mathbf{d}_{0})/\partial{\mu_{14}}$  & $\ \ \ 2.3126\times10^{-4}$  & $\ \ \ 2.0942\times10^{-4}$  & $\ \ \ 1.3000\times10^{-4}$ \tabularnewline
$\partial{P_{F}}(\mathbf{d}_{0})/\partial{\mu_{15}}$  & $-6.3125\times10^{-4}$  & $-6.2761\times10^{-4}$  & $-5.8333\times10^{-4}$ \tabularnewline
$\partial{P_{F}}(\mathbf{d}_{0})/\partial{\mu_{16}}$  & $\ \ \ 2.2333\times10^{-4}$  & $\ \ \ 2.2261\times10^{-4}$  & $\ \ \ 1.3333\times10^{-4}$ \tabularnewline
$\partial{P_{F}}(\mathbf{d}_{0})/\partial{\mu_{17}}$  & $-3.0844\times10^{-5}$  & $-3.9551\times10^{-5}$  & $0.0$ \tabularnewline
$\partial{P_{F}}(\mathbf{d}_{0})/\partial{\mu_{18}}$  & $-2.0729\times10^{-4}$  & $-2.6582\times10^{-4}$  & $-8.8000\times10^{-4}$ \tabularnewline
$\partial{P_{F}}(\mathbf{d}_{0})/\partial{\mu_{19}}$  & $-3.5881\times10^{-3}$  & $-3.4714\times10^{-3}$  & $-3.2900\times10^{-3}$ \tabularnewline
$\partial{P_{F}}(\mathbf{d}_{0})/\partial{\mu_{20}}$  & $-4.1604\times10^{-3}$  & $-4.0774\times10^{-3}$  & $-3.2200\times10^{-3}$ \tabularnewline
$\partial{P_{F}}(\mathbf{d}_{0})/\partial{\mu_{21}}$  & $-7.7002\times10^{-4}$  & $-7.2830\times10^{-4}$  & $-8.5000\times10^{-4}$ \tabularnewline
No. of FEA & $3445$ & $10^{7}$ & $22\times10^{7}$\tabularnewline
\hline
\end{tabular}
\renewcommand*\arraystretch{1.0}
\par\end{centering}
\end{table}

It is important to recognize that the PDD-SPA method can be applied
to solve this series-system reliability problem by interpreting the
failure domain as $\Omega_{F}:=\{\mathbf{x}:y_{s}(\mathbf{x})<0\}$,
where $y_{s}(\mathbf{X}):=\min\{y_{1}(\mathbf{X}),y_{2}(\mathbf{X})\}$
and then constructing a PDD approximation of $y_{s}(\mathbf{X})$.
In doing so, however, $y_{s}$ is no longer a smooth function of $\mathbf{X}$,
meaning that the convergence properties of the PDD-SPA method can
be significantly deteriorated. More importantly, the PDD-SPA method
is not suitable for a general system reliability problem involving
multiple, interdependent component performance functions. This is
the primary reason why the results of the PDD-SPA method are not included
in this example.

\subsection{Example 6: A Three-Hole Bracket}

The final example involves robust shape design optimization of a two-dimensional,
three-hole bracket, where nine random shape parameters, $X_{i}$,
$i=1,\cdots,9$, describe its inner and outer boundaries, while maintaining
symmetry about the central vertical axis. The design variables, $d_{k}={\displaystyle \mathbb{E}_{\mathbf{d}}}[X_{k}]$,
$k=1,\cdots,$9, are the means of these independent random variables,
with Figure \ref{fig:7ex6}(a) depicting the initial design of the
bracket geometry at the mean values of the shape parameters. The bottom
two holes are fixed, and a deterministic horizontal force $F=15,000$
N is applied at the center of the top hole. The bracket material has
a deterministic mass density $\rho=7810$ $\mathrm{kg}/\mathrm{m}^{3}$,
deterministic elastic modulus $E=207.4$ GPa, deterministic Poisson's
ratio $\nu=0.3$, and deterministic uniaxial yield strength $S_{y}=800$
MPa. The objective is to minimize the second-moment properties of
the mass of the bracket by changing the shape of the geometry such
that the maximum von Mises stress $\sigma_{e,\max}(\mathbf{X})$ does
not exceed the yield strength $S_{y}$ of the material with 99.875\%
probability if $y_{1}$ is Gaussian. Mathematically, the RDO problem
is defined to
\begin{equation}
\begin{array}{rcl}
{\displaystyle \min_{\mathbf{d}\in\mathcal{D}}c_{0}(\mathbf{d})} & = & {\displaystyle 0.5{\displaystyle \frac{{\displaystyle \mathbb{E}_{\mathbf{d}}}\left[y_{0}(\mathbf{X})\right]}{{\displaystyle \mathbb{E}_{\mathbf{d}_{0}}}\left[y_{0}(\mathbf{X})\right]}}+0.5\frac{\sqrt{\mathrm{var}_{\mathbf{d}}\left[y_{0}(\mathbf{X})\right]}}{\sqrt{\mathrm{var}_{\mathbf{d}_{0}}\left[y_{0}(\mathbf{X})\right]}},}\\
\mathrm{subject\: to\:}c_{1}(\mathbf{d}) & = & {\displaystyle 3\sqrt{\mathrm{\mathrm{var}_{\mathbf{d}}}\left[y_{1}(\mathbf{X})\right]}-{\displaystyle \mathbb{E}_{\mathbf{d}}}\left[y_{1}(\mathbf{X})\right]\le0},\\
\\
 &  & 0\;\mathrm{mm}\le d_{1}\le14\;\mathrm{mm},\;17\;\mathrm{mm}\le d_{2}\le35\;\mathrm{mm},\\
 &  & 10\;\mathrm{mm}\le d_{3}\le30\;\mathrm{mm},\;30\;\mathrm{mm}\le d_{4}\le40\;\mathrm{mm},\\
 &  & 12\;\mathrm{mm}\le d_{5}\le30\;\mathrm{mm},\;12\;\mathrm{mm}\le d_{6}\le30\;\mathrm{mm},\\
 &  & 50\;\mathrm{mm}\le d_{7}\le140\;\mathrm{mm},\;-15\;\mathrm{mm}\le d_{8}\le10\;\mathrm{mm},\\
 &  & -8\;\mathrm{mm}\le d_{9}\le15\;\mathrm{mm},
\end{array}\label{rdo}
\end{equation}
where $\mathbf{d}=(d_{1},\cdots,d_{9})\in\mathcal{D}\subset\mathbb{R}^{9}$
is the design vector;
\begin{equation}
y_{0}(\mathbf{X})=\rho\int_{\mathcal{D}'(\mathbf{X})}d\mathcal{D}'\label{44}
\end{equation}
and
\begin{equation}
y_{1}(\mathbf{X})=S_{y}-\sigma_{e,\max}(\mathbf{X})\label{45}
\end{equation}
are two random response functions; ${\displaystyle \mathbb{E}_{\mathbf{d}}}[y_{0}(\mathbf{X})]$
and ${\displaystyle \mathrm{var}_{\mathbf{d}}}[y_{0}(\mathbf{X})]:={\displaystyle \mathbb{E}_{\mathbf{d}}}[y_{0}(\mathbf{X})-{\displaystyle \mathbb{E}_{\mathbf{d}}}[y_{0}(\mathbf{X})]]^{2}$
are the mean and variance, respectively, of $y_{0}$ at design $\mathbf{d}$;
and ${\displaystyle \mathbb{E}_{\mathbf{d}}}[y_{1}(\mathbf{X})]$
and ${\displaystyle \mathrm{var}_{\mathbf{d}}}[y_{1}(\mathbf{X})]:={\displaystyle \mathbb{E}_{\mathbf{d}}}[y_{1}(\mathbf{X})-{\displaystyle \mathbb{E}_{\mathbf{d}}}[y_{1}(\mathbf{X})]]^{2}$
are the mean and variance, respectively, of $y_{1}$ at design $\mathbf{d}$.
The initial design $\mathbf{d}_{0}=\{0,30,10,40,20,20,75,0,0\}^{T}$ mm.
Figure \ref{fig:7ex6}(b) portrays the contours of the von Mises stress
calculated by the FEA of the initial bracket design, which comprises
11,908 nodes and 3914 eight-noded quadrilateral elements. A plane
stress condition was assumed. The approximate optimal solution is
denoted by $\tilde{\mathbf{d}}^{*}=\{\tilde{d}_{1}^{*},\cdots,\tilde{d}_{9}^{*}\}^{T}$.
\begin{figure}[tbph]
\begin{centering}
\includegraphics[clip,scale=0.73]{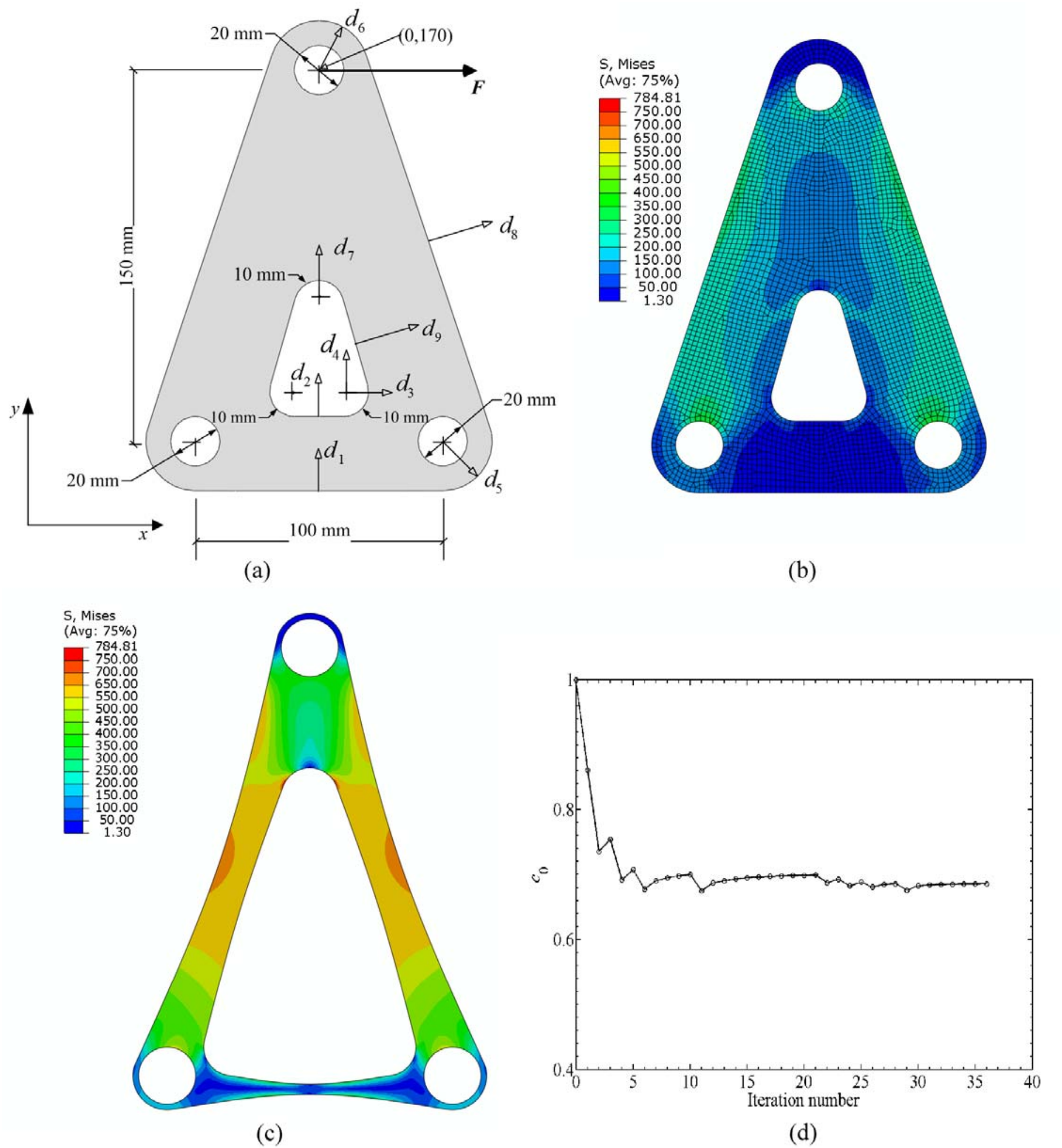}
\par\end{centering}
\caption{A three-hole bracket; (a) design parameterization; (b) von Mises stress
at the initial design; (c) von Mises stress at the final design; (d)
iteration history of the objective function (Example 6)}
\label{fig:7ex6}
\end{figure}

Due to their finite bounds, the random variables $X_{i}$, $i=1,\cdots,N$,
were assumed to follow truncated Gaussian distributions with densities
\begin{equation}
f_{X_{i}}(x_{i};\mathbf{d})=\left\{ \begin{array}{l@{\quad:\quad}l}
{\displaystyle \frac{1}{\Phi(D_{i})-\Phi(-D_{i})}}{\displaystyle \phi\left(\frac{x_{i}-d_{i}}{\sigma_{i}}\right)}\; & \alpha_{i}\le x_{i}\le\beta_{i},\\
0 & \mathrm{otherwise},
\end{array}\right.
\end{equation}
where $\sigma_{i}=0.2$; and $\alpha_{i}=d_{i}-D_{i}$
and $\beta_{i}=d_{i}+D_{i}$ are the lower and upper bounds, respectively,
of $X_{i}$. To avoid unrealistic designs, the bounds were chosen
as follows: $D_{i}=2$ for all $i=1,\cdots,9$. These conditions are
consistent with the bound constraints of design variables stated in
Equation \ref{rdo}.

A multi-point single-step PDD method \cite{ren12}, employing univariate
($S=1$), first-order ($m=1$) PDD approximation of the objective
and constraint functions and their design sensitivities from the proposed
method, was employed to solve this RDO problem. Since classical orthonormal
polynomials do not exist for truncated Gaussian distributions, again
the Stieltjes procedure was employed to determine the measure-consistent
orthonormal polynomials and corresponding Gauss quadrature formula
\cite{rahman09b}. The largest order $m'=2$ for the Fourier polynomial
expansions of the derivatives of log-density functions. The sensitivities
of the first two moments of $y_{0}(\mathbf{X})$ and $y_{1}(\mathbf{X})$,
required in the sequential quadratic optimization, were analytically
calculated from Equations \ref{mom1spdd2} and \ref{mom2spdd2}. Table
\ref{tab:7ex6-1} summarizes the optimization results, requiring 37
design iterations and 703 FEA to attain the final optimal design with
the corresponding mean shape presented in Figure \ref{fig:7ex6}(c).
The iteration history, depicted in Figure \ref{fig:7ex6}(d), indicates
rapid convergence due to accurate and efficient calculation of the
design sensitivities. Compared with the initial design in Figure \ref{fig:7ex6}(b),
the overall area of the optimal design has been substantially reduced,
mainly due to significant alteration of the inner boundary and moderate
alteration of the outer boundary of the bracket. All nine design variables
have undergone moderate to significant changes from their initial
values. The optimal mass of the bracket is $0.1207$ kg - about a
$65\%$ reduction from the initial mass of $0.3415$ kg. Due to robust design, the reduction of the mean is $65.1\%$, whereas the standard deviation diminishes by $4.4\%$. The smaller drop in the standard deviation is attributed to the objective function that combines both the mean and standard deviation of $y_{0}$.

\begin{table}[tbph]
\begin{centering}
\caption{Optimization results by the univariate PDD approximation ($S=1$,
$m=1$) (Example 6)}
\label{tab:7ex6-1}
\renewcommand*\arraystretch{1.2}
\begin{tabular}{ccc}
\hline
 & Initial design ($\mathbf{d}_{0}$) \  & Final design ($\mathbf{\tilde{d}}^{*}$) \tabularnewline
\hline
$\tilde{d}_{1}^{*}$, mm  & 0  & 13.4031 \tabularnewline
$\tilde{d}_{2}^{*}$, mm  & 30  & 17.0003 \tabularnewline
$\tilde{d}_{3}^{*}$, mm  & 10  & 27.1802 \tabularnewline
$\tilde{d}_{4}^{*}$, mm  & 40  & 30.0056 \tabularnewline
$\tilde{d}_{5}^{*}$, mm  & 20  & 12.0004 \tabularnewline
$\tilde{d}_{6}^{*}$, mm  & 20  & 12.0000 \tabularnewline
$\tilde{d}_{7}^{*}$, mm  & 75  & 118.035 \tabularnewline
$\tilde{d}_{8}^{*}$, mm  & 0  & -13.8359 \tabularnewline
$\tilde{d}_{9}^{*}$, mm  & 0  & 14.9785 \tabularnewline
$\tilde{c}_{0}(\tilde{\mathbf{d}}^{*})$  & 1  & 0.6858 \tabularnewline
$\tilde{c}_{1}(\tilde{\mathbf{d}}^{*})$ MPa  & -433.328  & -8.084 \tabularnewline
\hline
\end{tabular}
\renewcommand*\arraystretch{1.0}
\par\end{centering}
\end{table}

\section{Conclusions}

Three novel computational methods grounded in PDD were developed for
design sensitivity analysis of high-dimensional complex systems subject
to random input. The first method, capitalizing on a novel integration
of PDD and score functions, provides analytical expressions of approximate
design sensitivities of the first two moments that are mean-square
convergent. Applied to higher-order moments, the method also estimates
design sensitivities by two distinct options, depending on how the
high-dimensional integrations are performed. The second method, the
PDD-SPA method, integrates PDD, SPA, and score functions, leading
to analytical formulae for calculating design sensitivities of probability
distribution and component reliability. The third method, the PDD-MCS
method, also relevant to probability distribution or reliability analysis,
utilizes the embedded MCS of the PDD approximation and score functions.
Unlike the PDD-SPA method, however, the sensitivities in the PDD-MCS
method are estimated via efficient sampling of approximate stochastic
responses, thereby affording the method to address both component
and system reliability problems. Furthermore, the PDD-MCS method is
not influenced by any added approximations, involving calculations
of the saddlepoint and higher-order moments, of the PDD-SPA method.
For all three methods developed, both the statistical moments or failure
probabilities and their design sensitivities are determined concurrently
from a single stochastic analysis or simulation. Numerical results
from mathematical examples corroborate fast convergence of the sensitivities
of the first two moments. The same condition holds for the sensitivities
of the tails of probability distributions when orthonormal polynomials
are constructed consistent with the probability measure of random
variables. Otherwise, the convergence properties may markedly degrade
or even disappear when resorting to commonly used transformations.
For calculating the sensitivities of reliability, the PDD-MCS method,
especially its bivariate version, provides excellent solutions to
all problems, including a 100-dimensional mathematical function, examined.
In contrast, the PDD-SPA method also generates very good estimates
of the sensitivities, but mostly for small to moderate uncertainties
of random input. When the coefficient of variation is large, the PDD-SPA
method may produce inaccurate results, suggesting a need for further
improvements. Finally, a successful application on robust design optimization
of a three-hole bracket demonstrates the usefulness of the methods
developed.

The computational effort of the univariate PDD method varies linearly
with respect to the number of random variables and, therefore, the
univariate method is highly economical. In contrast, the bivariate
PDD method, which generally outperforms the univariate PDD method,
demands a quadratic cost scaling, making it also more expensive than
the latter method. Nonetheless, both versions of the PDD method are
substantially more efficient than crude MCS.

\ack The authors acknowledge financial support from the U.S. National Science
Foundation under Grant No. CMMI-0969044.

\end{document}